\documentclass[11pt]{article}
\usepackage{amsmath,mathrsfs}
\usepackage{amssymb}
\usepackage{amsthm}
\usepackage[all]{xy}
\usepackage{indentfirst}
\usepackage{color}
\allowdisplaybreaks

\newtheorem{theorem}{Theorem}[section]
\newtheorem*{theorem*}{Theorem}
\newtheorem{corollary}[theorem]{Corollary}
\newtheorem{assumption}{Assumption}
\newtheorem{remark}[theorem]{Remark}
\newtheorem{lemma}[theorem]{Lemma}
\newtheorem{definition}[theorem]{Definition}

\newtheorem{proposition}[theorem]{Proposition}
\newtheorem{example}[theorem]{Example}
\newtheorem*{proposition*}{Proposition}

\newcommand{\R}{\mathbb{R}}

\newcommand{\be}{\begin{eqnarray*}}
\newcommand{\ee}{\end{eqnarray*}}
\newcommand{\ba}{\begin{align*}}
\newcommand{\bpm}{\begin{pmatrix}}
\newcommand{\epm}{\end{pmatrix}}

\newcommand{\bx}{\boldsymbol{x}}
\newcommand{\bn}{\boldsymbol{n}}

\newcommand{\by}{\boldsymbol{y}}

\begin{document}

\title{On the Fueter-Sce theorem and  Cauchy-Kovalevskaya extensions over alternative $\ast$-algebras}
\author{Qinghai Huo$^1$\thanks{This work was partially supported by the National Natural Science Foundation of China (No. 12301097),   Fundamental Research Funds for the Central Universities (No. JZ2025HGTB0171) and China Scholarship Council (No. 202506690055).}
,\ Irene Sabadini$^2$\thanks{This work was partially supported by PRIN 2022 {\em Real and Complex Manifolds: Geometry and Holomorphic Dynamics. I.S. is a member of GNSAGA of INdAM}.}\ ,
 and Zhenghua Xu$^1$\thanks{This work was partially supported by  Anhui Provincial Natural Science Foundation (No. 2308085MA04), Fundamental Research Funds for the Central Universities (No. JZ2025HGTG0250) and China Scholarship Council (No. 202506690052).}\\
 \\
\emph{$^1$\small School of Mathematics, Hefei University of Technology,}
\emph{\small  Hefei, 230601, P.R. China}\\
\emph{\small E-mail address:  hqh86@mail.ustc.edu.cn;  zhxu@hfut.edu.cn}
\\
\emph{\small $^2$Dipartimento di Matematica, Politecnico di Milano,}
\emph{\small Via E. Bonardi, 9, 20133 Milano, Italy} \\
\emph{\small E-mail address:   irene.sabadini@polimi.it}
}


\maketitle

\begin{abstract}
Recently,   the concept of generalized partial-slice monogenic (or regular) functions has been introduced  and studied over Clifford algebras and  octonions, respectively.  In this paper,  we further  develop  the theory of generalized partial-slice monogenic  functions defined on hypercomplex subspaces and with values in  a real alternative $\ast$-algebra and we concentrate on the Fueter-Sce  theorem,  three types of Cauchy-Kovalevskaya  extensions, and their various  internal relationships. The paper proposes more  bridges between the  theories of monogenicity, harmonicity, and   generalized partial-slice monogenicity.
 \end{abstract}
\tableofcontents
\section{Introduction}
Hypercomplex analysis is a higher-dimensional  theory of holomorphicity generalizing the theory of holomorphic functions from  the algebra of complex numbers to more general algebras, such as quaternions, octonions, bicomplex numbers, and Clifford algebras, see e.g. \cite{Colombo-Sabadini-Struppa-20} for some historical remarks. Among the various possibilities, two theories have been proven to be particularly successful: the theory of monogenic (also called   regular,  hyperholomorphic) functions and of slice-monogenic  (called slice-regular, slice-hyperholomorphic) functions.

The theory of regular functions was initiated by Moisil and Fueter  \cite{Fueter} and  fully investigated by Fueter's school. It was further generalized to functions with values in a Clifford algebra in the kernel of the Dirac or of the Weyl operator. There is wide literature on monogenic (or regular) functions, see e.g. \cite{Brackx,Colombo-Sabadini-Sommen-Struppa-04,Delanghe-Sommen-Soucek-92,Gilbert,Gurlebeck} and references therein.
As is well known, a fundamental limitation of this theory is that powers of the quaternionic variable are not regular, as they do not belong to the kernel of the chosen operator.

The theory of slice regular functions of one quaternionic variable was initiated by Gentili and Struppa \cite{Gentili-Struppa-07},   inspired by an   earlier idea due to Cullen.  This function theory includes power series of the form
 $\sum_n q^na_n$, where  $q$ is the quaternionic variable and $a_n$ are constant and both lie  in the algebra  of quaternions. Thereafter, this function class was extended to more general cases, such as Clifford algebras \cite{Colombo-Sabadini-Struppa-09},  octonions \cite{Gentili-Struppa-10},  real alternative  $*$-algebras \cite{Ghiloni-Perotti-11}. In the past twenty years, slice analysis  has gain  great interest in both theory  \cite{Colombo-Sabadini-Struppa-11,Gentili-Stoppato-Struppa-22} and its applications, especially in  functional calculus for noncommutative operators \cite{Alpay,CDPS-23,CDPS-23A,Colombo-Gantner-20,Colombo-Gantner-Kimsey-18,Wand-22,Wand-25}. Interested readers can refer to  \cite{Altavilla-18,Gentili-Salamon-Stoppato-14,Ghiloni-Perotti-Stoppato-22} for applications  of slice regular functions of one quaternion (or octonion) variable to the study of twistor transforms and  orthogonal complex structures.

Both the two types of function theories have advantages and drawbacks.
In 2023, the concept of \textit{generalized partial-slice monogenic functions}  was  introduced  over Clifford algebras \cite{Xu-Sabadini}.
 The function theory includes the two theories of monogenic functions  and of slice monogenic functions as special cases.  See   \cite{Ding-Xu,Huo-24,Xu-Sabadini-2,Xu-Sabadini-3,Xu-Sabadini-4} for more results in this new setting. As we all know, all Clifford algebras  are   associative, so it is natural to  further investigate generalized partial-slice monogenic functions  in the more general setting of real alternative $\ast$-algebras, in particular in the case of octonions which was  studied in \cite{Xu-Sabadini-26}.

In \cite{Ghiloni-Stoppato-24-1,Ghiloni-Stoppato-24-2}, the concept of partial-sliceness has been developed by Ghiloni and Stoppato by introducing the theory of $T$-regular functions, which  proposes a unified theory of regularity in one hypercomplex variable for an alternative $\ast$-algebra.   Recently,   the Fueter-Sce phenomenon for $T$-regular functions over general associative $\ast$-algebras has been studied \cite{Ghiloni-Stoppato-25}. Note that  all results in \cite{Ghiloni-Stoppato-24-1,Ghiloni-Stoppato-24-2,Ghiloni-Stoppato-25}, except the representation formula (and its consequences), pertain to associative algebras.
For further results concerning functions over (potentially non-associative) alternative
$*$-algebras, the reader may refer to e.g. \cite{Ghiloni-Perotti-Stoppato-17adv,Perotti-22} for the slice regular case,  or e.g. \cite{HRX,Xu-Ding-Wang} for  monogenic functions, and \cite{BP-26} for the recently introduced Dunkl-regular functions.

This paper is a continuation of our works \cite{Xu-Sabadini,Xu-Sabadini-26}. We shall  consider  generalized partial-slice monogenic functions over  general alternative $\ast$-algebras and  concentrate on the Fueter-Sce  theorem and three types of Cauchy-Kovalevskaya extensions (CK-extensions for short).   The main innovation points of this article are reflected in two aspects. The first aspect is that the results on the Fueter-Sce mapping are obtained  for  non-associative algebras, rather than remaining within the realm of associative algebras in \cite{Ghiloni-Stoppato-25,Xu-Sabadini-2}. The second aspect is that we extend the results on CK-extensions in \cite{Brackx,Delanghe-Sommen-Soucek-92,De-Adan}  to  our  new   framework of generalized partial-slice monogenic functions over general alternative $\ast$-algebras. The non-associativity emerges in  various statements and proofs and require specific techniques to obtain the results.

The paper is structured as follows. In Section 2,  we recall some basic definitions  on real alternative $\ast$-algebras which are the general framework of this paper. Then we introduce the notion of  generalized partial-slice monogenic functions on hypercomplex subspaces.  In particular,   we provide  some examples as well as the fundamental properties of this function class,  such as  identity theorem (Theorem \ref{Identity-theorem}),   Representation Formula (Theorem \ref{Representation-Formula-SM}),  Cauchy   integral formula (Theorem \ref{Cauchy-slice}), and Fueter polynomials (Definition \ref{definition-Fueter}).

In Section 3,    we  develop  three types of  CK-extensions  from  real analytic functions defined in   domains in $\mathbb{R}^{p+1}$  and valued in  real alternative $\ast$-algebras, to generalized partial-slice monogenic functions  (see Theorem \ref{slice-Cauchy-Kovalevsky-extension}), monogenic functions
(see Theorem \ref{generalized-CK-extension}), and harmonic functions (see Theorem \ref{Harm-generalized-CK-extension}), respectively.
Furthermore, we  introduce the notions of  \textit{partial even part} and  \textit{partial odd part}, see Definition \ref{even-odd-parts}, and   establish some relationships between the generalized CK-extensions for monogenic and  harmonic functions in Corollary \ref{GCK-HGCK} and Proposition \ref{D-GCK-HGCK}. These relations are useful to find more connections between  the Fueter-Sce theorem and generalized CK-extensions, see Section  5.

The main result in  Section  4 is the Fueter-Sce theorem  for generalized partial-slice regular  functions on  partially symmetric  domains in hypercomplex subspaces; see   Theorem \ref{Fueter-theorem}.

In  Section  5, we first establish the  connection between  the Fueter-Sce theorem and generalized CK-extension for monogenic  functions, see Theorem \ref{CK-Fueter-relation-M}, which generalizes  \cite[Theorem 4.7]{Xu-Sabadini-2} from  Clifford algebras to general  alternative $\ast$-algebras.
Furthermore, using the notions of   partial even and odd parts,   we prove more  relations between  the Fueter-Sce theorem and three types of CK-extensions   obtained in Section 3; see  Theorems \ref{CK-Fueter-relation-M-H} and \ref{CK-Fueter-relation-H}.

Finally, in  Section  6,  we give some avenues for further research by discussing the class of poly-monogenic functions (and  generalized partial-slice  monogenic functions), and   poly-Dunkl-monogenic functions for Dunkl operators associated with the Coxeter group, over real alternative $\ast$-algebras.

\section{Preliminary results}
In this section, we collect some preliminary results on alternative $\ast$-algebras  and  then introduce the notion  of  generalized partial-slice monogenic functions on   hypercomplex subspaces, together with some examples and  fundamental properties such as identity theorem, Representation Formula,
Cauchy integral formula, and  Fueter polynomials. For more details, we refer the reader to  \cite{Okubo,Schafer,Gentili-Struppa-10,Ghiloni-Stoppato-24-2,Perotti-22}.
\subsection{Real alternative $\ast$-algebras}
\noindent Let $\mathbb{A}$ be a real algebra with the unity $1$. A real algebra $\mathbb{A}$ is said to be alternative if the \textit{associator}
$[a,b,c]:= (ab)c-a(bc)$ is  a trilinear and alternating function in   $a,b,c\in \mathbb{A}$.
All  alternative algebras   obey the following   associative laws:
\begin{itemize}\item
 Artin's   theorem:
\textit{the subalgebra generated by two elements of $\mathbb  A$ is associative}.
\item
Moufang identities:
$$a(b(ac)) = (aba)c,\qquad ((ab)c)b = a(bcb), \qquad (ab)(ca) = a(bc)a,$$
for  $a,b,c\in\mathbb A$.\end{itemize}

As  immediate consequences of the Artin's  theorem, we have
\begin{proposition}\label{artin-inverse}(\cite[p. 38]{Schafer})
For any $x, y \in \mathbb{A},$ if x is invertible, then it holds that
$$[x^{-1},x,y]=0.$$
\end{proposition}
\begin{proposition}\label{real}
For any $r\in \mathbb{R}$ and $x, y \in \mathbb{A},$ it holds that
$$[r,x,y]=0.$$
\end{proposition}

A real algebra $\mathbb  A$  is called   a  $\ast$-algebra  if $\mathbb  A$  is equipped with an anti-involution (also called $\ast$-involution), which is a real linear map $^{c}:\mathbb  A \rightarrow \mathbb  A$,  $a\mapsto a^{c}$ satisfying
$$a^{c}= a, \quad a\in \mathbb{R},$$
and
$$(a^{c})^{c} = a,  \ (ab)^{c} =b^{c}a^{c},  \quad a,b \in\mathbb A.$$
\begin{assumption}
Let $(\mathbb A, +,  \cdot, ^{c})$ be an alternative real  $\ast$-algebra with the unity $1$,  of finite dimension $d>1$ as   a real vector space, and  equipped with an anti-involution $^{c}$.
Additionally, we endow  $\mathbb A$   with the natural topology and the differential structure as a real vector space.
\end{assumption}

\begin{example} [Division algebras]
The   division algebras of  the complex numbers $\mathbb{C}$, quaternions $\mathbb{H}$ and  octonions $\mathbb{O}$ are real  $\ast$-algebras, where $\ast$-involutions are  the standard conjugations of  complex numbers, quaternions, and  octonions, respectively.
\end{example}
\begin{example} [Clifford algebras]\label{example-Clifford}
 The   Clifford algebra   $\mathbb{R}_{0,m}$ is an associative $\ast$-algebra,  which is generated  by the standard orthonormal basis $\{e_1,e_2,\ldots, e_m\}$ of the $m$-dimensional real Euclidean space  $\mathbb{R}^m$ by assuming   $$e_i e_j + e_j e_i= -2\delta_{ij}, \quad 1\leq i,j\leq m,$$
where    $\ast$-involution is given by the standard Clifford conjugation.
\end{example}

\subsection{Hypercomplex subspaces}
 Let $\mathbb A$ be a $\ast$-algebra. For   $x\in \mathbb A$, its \textit{trace}  is  define  by $t(x):= x+x^{c}\in {\mathbb A}$ and its  (squared) \textit{norm} is $n(x):= xx^{c}\in {\mathbb A}.$ Denote  the  sphere  of the imaginary units of $\mathbb A$ compatible with the  $\ast$-algebra structure of $\mathbb A$ by
 $$\mathbb{S}_{{\mathbb A}}:= \{x  \in \mathbb A : t(x)=0,  n(x)=1  \}.$$
For each $J \in \mathbb{S}_{{\mathbb A}}$,   denote by $$\mathbb{C}_{J}:=\langle1,J\rangle \cong \mathbb{C},$$ the subalgebra of ${\mathbb A}$ generated by $1$ and $J$.  Obviously, it holds that
 $$\mathbb C_I \cap \mathbb C_J=\mathbb R, \qquad  I, J\in\mathbb S_{\mathbb A}, \ I\neq \pm J.$$
Define the so-called quadratic cone of ${\mathbb A}$ by
$$Q_{{\mathbb A}}:= \mathbb{R} \cup \big\{x \in {\mathbb A} \mid t(x) \in  \mathbb{R}, \ n(x)\in \mathbb{R}, \  4\,  n(x)>t(x)^{2} \big\}.$$
It was proved, see \cite{Ghiloni-Perotti-11}, that
$$  Q_{\mathbb A}=\bigcup_{J\in\mathbb S_{\mathbb A}} \mathbb C_J.$$

\begin{assumption}
Assume  $\mathbb{S}_{{\mathbb A}}\neq \emptyset$.
\end{assumption}

\begin{definition}
Let $M$ be a real vector subspace of the $\ast$-algebra $\mathbb A$. An ordered real vector
basis $(v_0,v_1,  \ldots, v_m)$  of $M $ is called a hypercomplex basis of $M$ if: $m \geq 1$; $v_0 =1$; $v_s \in \mathbb{S}_{{\mathbb A}}$
and $v_sv_t = -v_tv_s$ for all distinct $s, t\in \{1,  \ldots, m\}$. The subspace $M$ is called a {\em hypercomplex
subspace} of $\mathbb A$ if $\mathbb{R}  \subsetneq   M  \subseteq Q_{{\mathbb A}}$.
\end{definition}

Equivalently, a basis $(1,v_1,  \ldots, v_m)$ is a hypercomplex basis if, and only if, $t(v_s)=0, n(v_s)=1$
and $t(v_s v_t^{c})=0$ for all distinct $s, t\in \{1,\ldots , m\}$.

\begin{assumption}\label{assumption}
 Within the real alternative $\ast$-algebra $\mathbb A$, we fix a hypercomplex subspace $M$, with a fixed hypercomplex basis  $\mathcal{B}=(v_0,v_1,  \ldots, v_m)$ with $v_0=1$. Moreover, $\mathcal{B}$ is completed to a real vector basis $\mathcal{B}' =(v_0, v_1,  \ldots , v_d)$ of  $\mathbb A$ and  $\mathbb A$ is endowed with the standard Euclidean scalar product $\langle\cdot,\cdot\rangle$ and norm $ | \cdot |$ associated to $\mathcal{B}'$. If this is the case, then we have for all $x, y\in M$
$$t(xy^{c})=t(y^{c}x)=2 \langle x,y\rangle,$$
$$n(x)=n(x^c)=|x|^{2}.$$
 \end{assumption}

\begin{example}\label{example-hypercomplex-subspaces}
 The best-known examples of    hypercomplex subspaces   in real alternative $\ast$-algebras   are given by
\begin{eqnarray*}
M=
\begin{cases}
\mathbb H \ \rm{or}  \  \mathbb H_r,  \    \ \, \mathbb A=\mathbb H, \qquad      Q_{\mathbb A}=\mathbb H,
\\
\mathbb O, \qquad  \quad  \ \, \mathbb A=\mathbb O, \qquad     Q_{\mathbb A}=\mathbb O,
\\
    \mathbb{R}^{m+1}, \qquad  \mathbb A=\mathbb{R}_{0, m}, \ \    Q_{\mathbb A}\supseteq\mathbb R^{m+1},
\end{cases}
\end{eqnarray*}
where  $\mathbb H_r$ denotes the three-dimensional subspace  of reduced quaternions
$\{x=x_0+x_1i+x_2j\in \mathbb{H}: x_0,x_1,x_2\in \mathbb{R}\}$
and $\mathbb R^{m+1}$ denotes the space of paravectors in $\mathbb{R}_{0, m}$.
\end{example}

In this paper, we shall  need the following useful property which is  a consequence of  the Artin's  theorem.
  \begin{proposition}\label{artin}
For any $x\in M$ and $y \in \mathbb{A},$ it holds that
$$[x,x,y]=[x^{c},x,y]=0.$$
\end{proposition}

\subsection{Generalized partial-slice monogenic functions}
In this  subsection, we  formulate  the concept  of  generalized partial-slice monogenic functions  with values in a real alternative $\ast$-algebra, and defined on hypercomplex
subspaces of this algebra.

Throughout the paper,  any vector $(x_0,x_1,\ldots,x_m)\in\R^{m+1}$ will be identified with the element  $x\in M$ via
$$(x_0,x_1,\ldots,x_n) \rightarrow x=x_0+\sum_{s=1}^{m}x_sv_s. $$
Denote $\mathbb{N}=\{0,1,2,\ldots\}$. For $k\in \mathbb{N}\cup \{\infty,\omega\}$,  denote by $ C^{k}(\Omega,\mathbb{A})$ the set of all functions  $f=\sum_{t=0}^{d}  v_{t}   f_{t}$ with real-valued components $f_{t}\in  C^{k}(\Omega)$, where  $\Omega$ is an open  set in $M$. In particular,    $C^{\omega}(\Omega, \mathbb{A})$    denotes the bilateral $\mathbb{A}$-module of real analytic functions on $\Omega$.

 For $s=0,1,\ldots,m$, let
$$\frac{\partial f}{\partial x_{s}}(x) = \lim_{\mathbb{R}\ni\varepsilon\rightarrow0} \varepsilon^{-1} (f(x+\varepsilon v_s)-f(x)).$$
Note that the operator $\partial_{x_{s}}=\frac{\partial}{\partial x_s} (s\geq1)$  does not depend on the  basis $\mathcal{B}'$ of $\mathbb{A}$, but only
 on the choice of $\mathcal{B}$. Moreover, the usual Leibniz rule holds:
 $$ \partial_{x_{s}} (fg)=(\partial_{x_{s}}f)g+f(\partial_{x_{s}}g), \quad f,g\in C^{1}(\Omega,\mathbb{A}),$$
which gives that
  $$ \partial^{2}_{x_{s}} (fg)=(\partial_{x_{s}}^{2}f)g+2 (\partial_{x_{s}} f)(\partial_{x_{s}} g) + f(\partial_{x_{s}}^{2}g), \quad f,g\in C^{2}(\Omega,\mathbb{A}).$$

Consider the  generalized   Cauchy-Riemann operator (or  Weyl operator)  induced by $\mathcal{B}$,
$$D_{\mathcal{B}}=  \sum_{s=0}^{m}v_s\partial_{x_{s}},$$  which corresponds to the Cauchy-Riemann operator (or  Weyl operator) giving monogenic functions defined on paravectors   and $\mathbb{R}_{0,m}$-valued, see e.g. \cite{Brackx}, and generalized first to octonions   by Dentoni and Sce \cite{Dentoni-Sce}, and then to real alternative $*$-algebras by Perotti in \cite[Definition 2]{Perotti-22}  where the symbol $\overline{\partial}_{\mathcal{B}}$ $(=\frac{1}{2}D_{\mathcal{B}})$ is used.

 In view of the non-commutativity of $\mathbb{A}$,   the action on functions $ f=\sum_{t=0}^{d}  v_{t}   f_{t}\in C^{1}(\Omega,\mathbb{A})$, can be on the left
$$ D_{\mathcal{B}} f(x):=\sum _{s=0}^{m}v_{s} \partial_{x_{s}} f(x)
= \sum _{s =0}^{m} \sum _{t=0}^{d}   v_{s}v_{t}\partial_{x_{s}} f_{t} (x),  $$
or on the  right
$$ f(x)D_{\mathcal{B}}:=\sum _{s=0}^{m}   \partial _{x_{s}} f(x)v_{s}= \sum _{s =0}^{m} \sum _{t=0}^{d}   v_{t}v_{s}
 \partial _{x_{s}} f_{t} (x).$$

In fact, the operator $D_{\mathcal{B}}$ does not depend on the choice of the  hypercomplex basis $\mathcal{B}$ of $M$ \cite[Remark 3]{BP-26}, which was proved in 2014  for the octonionic case \cite[Theorem 2.1]{LW}.
 Hence we shall use the symbol   $D$  (or $D_{x}$), instead of  $D_{\mathcal{B}}$.  Similarly, also the conjugated  generalized   Cauchy-Riemann operator   induced by $\mathcal{B}$
$$\overline{D}=\overline{D}_{x}= \sum_{s=0}^{m}v_s^{c}\partial_{x_{s}}=\partial_{x_{0}}-\sum_{s=1}^{m}v_s\partial_{x_{s}}$$
does not depend on the choice of the  hypercomplex basis $\mathcal{B}$.

Now we recall a useful property from \cite[Proposition 5]{Perotti-22}.
\begin{lemma}\label{d-dbar-laplace}
Let $\Omega $ be an open set   in $M$ and  $f\in C^{2}(\Omega, \mathbb{A})$. Then it holds that
  $$\Delta f = \overline{D} ( D f)=  D (\overline{D} f)=( f D) \overline{D}=( f \overline{D}) D,$$
where $\Delta=\Delta_{x}$  is the Laplacian  on $M$  induced by $\mathcal{B}$ given by
$$\Delta f(x)=\sum_{s=0}^{m} \partial_{x_{s}} (\partial_{x_{s}}f(x)).$$
\end{lemma}
Note that the Laplacian $\Delta $,  being an operator with real coefficients, obviously does not depend  on the choice of the  hypercomplex basis $\mathcal{B}$ of $M$.

 From now on,  let $p\in  \{0,1,\ldots,  m-1\}$  and set $q=m-p$. We shall split the element $x\in M$ into
$$\bx=\bx_p+\underline{\bx}_q \in\R^{p+1}\oplus\R^{q}, \quad \bx_p=\sum_{s=0}^{p}x_s v_s,\ \underline{\bx}_q=\sum_{s=p+1}^{m}x_s v_s,$$
where  we use   $\bx$ in bold to emphasize this fixed splitting, instead of $x$.
Correspondingly, the generalized Cauchy-Riemann operator   is split as
\begin{equation}\label{Dxx}
D_{\bx}=D_{\bx_p}+D_{\underline{\bx}_q}, \quad D_{\bx_p}=\sum_{s=0}^{p}v_s\partial_{x_s}, D_{\underline{\bx}_q}=
\sum_{s=p+1}^{m}v_s\partial_{x_s}.
\end{equation}

Denote the sphere   in $\mathbb R^q$ by $\mathbb{S}$
$$\mathbb{S}=\big\{\underline{\bx}_q: \underline{\bx}_q^2 =-1\big\}=\Big\{\underline{\bx}_q=\sum_{s=p+1}^{m}x_s v_s:\sum_{s=p+1}^{m}x_s^{2}=1 \Big\}.$$
For $\underline{\bx}_q\neq0$, there exists a uniquely determined $r\in \mathbb{R}^{+}=\{x\in \mathbb{R}: x>0\}$ and $\underline{\omega}\in \mathbb{S}$, such that $\underline{\bx}_q=r\underline{\omega}$, where $r=|\underline{\bx}_q|, \underline{\omega}=\frac{\underline{\bx}_q}{|\underline{\bx}_q|}. $
For $\underline{\bx}_q= 0$, we set $r=0$ and $\underline{\omega}$ is any $\underline{\omega}\in \mathbb{S}$. For any $\bx=\bx_p+\underline{\bx}_q \in M$, we set $\bx':=(\bx_p,r)=(x_0,x_1,\ldots, x_p,r)\in \mathbb{R}^{p+2}$ with $r=|\underline{\bx}_q|$.

The upper half-space $\mathrm{H}_{\underline{\omega}}$ in $\mathbb{R}^{p+2}$ associated with $\underline{\omega}\in \mathbb{S}$ is defined by
$$\mathrm{H}_{\underline{\omega}}=\{\bx_p+r\underline{\omega}, \bx_p \in\R^{p+1}, r\geq0 \},$$
and it is clear that
$$ M=\bigcup_{\underline{\omega}\in \mathbb{S}} \mathrm{H}_{\underline{\omega}}, \quad \R^{p+1}=\bigcap_{\underline{\omega}\in \mathbb{S}} \mathrm{H}_{\underline{\omega}}.$$

Recalling    \eqref{Dxx},    the  definition of  generalized partial-slice monogenic functions over
 real alternative $\ast$-algebras can be given as follows.
\begin{definition} \label{definition-slice-monogenic}
 Let $\Omega$ be an open set in $M$. A function $f:\Omega \rightarrow \mathbb{A}$ is called left  generalized partial-slice monogenic of type $(p,q)$ if, for all $ \underline{\omega} \in \mathbb S$, its restriction $f_{\underline{\omega}}$ to $\Omega_{\underline{\omega}}:= \Omega\cap (\mathbb{R}^{p+1} \oplus \underline{\omega} \mathbb{R})\subseteq \mathbb{R}^{p+2}$  is of class $C^{1}$ and  satisfies
$$D_{\underline{\omega}}f_{\underline{\omega}}(\bx):=(D_{\bx_p}+\underline{\omega}\partial_{r}) f_{\underline{\omega}}(\bx_p+r\underline{\omega})=0,$$
for all $\bx=\bx_p+r\underline{\omega} \in \Omega_{\underline{\omega}}$.\\
Similarly,  $f$ is called   right  generalized partial-slice monogenic functions of type $(p,q)$  if, for all $ \underline{\omega} \in \mathbb S$, $f_{\underline{\omega}}\in C^1(\Omega_{\underline{\omega}}, \mathbb{A}) $ satisfies
$$f_{\underline{\omega}}(\bx)D_{\underline{\omega}}:={f_{\underline{\omega}} (\bx_p+r\underline{\omega})D_{\bx_p}}+ \partial_{r}f_{\underline{\omega}} (\bx_p+r\underline{\omega})\underline{\omega}=0,$$
for all $\bx=\bx_p+r\underline{\omega} \in \Omega_{\underline{\omega}}$.
 \end{definition}
Throughout this paper, $(p,q)$ is fixed.  Hence,  we denote by $\mathcal {GSM}^{L}(\Omega,\mathbb{A})$ (or $\mathcal {GSM}(\Omega,\mathbb{A})$ for short) and $\mathcal {GSM}^{R}(\Omega,\mathbb{A})$ the class of  all  left and right generalized partial-slice monogenic functions of type $(p,q)$ in $\Omega$, respectively.
  When $\mathbb{A}$ is  associative, such as   $\mathbb{A}=\mathbb{H}$ or $\mathbb{R}_{0,m}$,      $\mathcal{GSM}^{L}(\Omega, \mathbb{A})$ is   a right $\mathbb{A}$-module.
Meanwhile, when  $\mathbb{A}$ is     non-associative, e.g. for $\mathbb{A}=\mathbb{O}$,  $\mathcal{GSM}^{L}(\Omega, \mathbb{A})$ is not, in general, a right   $\mathbb{A}$-module.

When $(p,q)=(m-1,1)$, the  notion of (left) generalized partial-slice monogenic functions in Definition  \ref{definition-slice-monogenic} coincides with the notion of   (left) monogenic  functions,  denoted by $\mathcal {M}(\Omega,\mathbb{A})$. When $(p,q)=(0,m)$, Definition  \ref{definition-slice-monogenic} degenerates into the notion of (left) slice monogenic functions, see \cite{Colombo-Sabadini-Struppa-09}, denoted by  $\mathcal {SM}(\Omega,\mathbb{A})$.

\begin{example}\label{example-1}
Let $p\in\{0,1,\ldots,m-2\}, \underline{\omega} \in\mathbb{S}$ and $a, b\in \mathbb{A}$ with $a\neq b$.  Let $U=\mathrm{H}_{\underline{\omega}}\cup \mathrm{H}_{-\underline{\omega}}$ and consider
\begin{eqnarray*}
f(\bx)=
\left\{
\begin{array}{ll}
a,     &\mathrm {if} \ \bx \in  M \setminus U,
\\
b,   &\mathrm {if} \ \bx \in U \setminus \mathbb{R}^{p+1}.
\end{array}
\right.
\end{eqnarray*}
 Then $f \in \mathcal {GSM}^{L}(M\setminus \mathbb{R}^{p+1}, \mathbb{A}) \cap{\mathcal{GSM}}^{R}(M \setminus \mathbb{R}^{p+1}, \mathbb{A}).$
\end{example}
\begin{example}Given $n\in \mathbb{N}$, set
$$f(\bx)=(x_{0}+\underline{\bx}_{q})^{n}.$$
Then $f \in \mathcal {GSM}^{L}(M, \mathbb{A}) \cap{\mathcal{GSM}}^{R}(M,\mathbb{A}).$
\end{example}
\begin{example}\label{Cauchy-kernel-example}
Define the Cauchy kernel
$$E(\bx):=\frac{1}{\sigma_{p+1}}\frac{\overline{\bx}}{|\bx|^{p+2}}, \quad \bx \in\Omega=M\setminus \{0\},$$
where $\sigma_{p+1}=2\frac{\Gamma^{p+2}(\frac{1}{2}) }{\Gamma (\frac{p+2}{2})} $ is the surface area of the unit ball in $\mathbb{R}^{p+2}$. \\
Then $E(\bx) \in \mathcal {GSM}^{L}(\Omega, \mathbb{A}) \cap{\mathcal{GSM}}^{R}(\Omega, \mathbb{A}).$
\end{example}

 The following Cauchy integral formula can be proved using exactly the same arguments in the proof of \cite[Theorem 4.7]{Xu-Sabadini-26} and thus we omit the details.
\begin{theorem}[Cauchy formula]\label{Cauchy-slice}
Let $\underline{\eta}\in \mathbb{S}$ and     $\Omega$ be a bounded domain in $M$ with smooth boundary $\partial \Omega_{\underline{\eta}}$.  If $ f\in C^1(\Omega_{\underline{\eta}}, \mathbb{A}) \cap C(\overline{\Omega_{\underline{\eta}}}, \mathbb{A}) $  satisfies  $D_{\underline{\eta}}f(\by)=0$ for all $\by \in \Omega_{\underline{\eta}}$, then
 $$f(\bx)=\int_{\partial \Omega_{\underline{\eta}}}  E_{\by}(\bx) (\bn(\by)f(\by)) dS(\by), \quad  \bx \in \Omega_{\underline{\eta}},   $$
where $E_{\boldsymbol{y}}(\boldsymbol{x}):=E(\boldsymbol{y}-\boldsymbol{x})$, $\bn(\by)=\sum_{i=0}^{p}n_i (\by) v_i+n_{p+1}(\by) \underline{\eta}$ is the unit exterior normal to $\partial \Omega_{\underline{\eta}}$ at $\by$,
 $dS$ and  $dV$ stand  for  the   classical   Lebesgue surface  element and volume element  in $\mathbb{R}^{p+2}$, respectively.
\end{theorem}

\begin{remark}\label{C1-C00}{\rm
By Theorem \ref{Cauchy-slice}, all functions $f\in \mathcal {GSM}(\Omega, \mathbb{A})$ are such that    $f_{\underline{\omega} }\in C^{\infty}(\Omega_{\underline{\omega}}, \mathbb{A})$ for all $\underline{\omega}\in \mathbb S$. }  \end{remark}
From Remark \ref{C1-C00} and  Lemma \ref{d-dbar-laplace}, we get
 \begin{proposition}\label{GSM-Harm}
A function $f\in \mathcal {GSM}(\Omega, \mathbb{A})$  is necessarily harmonic slice-by-slice on $\Omega$, that is  $f_{\underline{\omega}}$  is harmonic in $\Omega_{\underline{\omega}} $ for any $\underline{\omega}\in \mathbb S$.
\end{proposition}

\subsection{Identity theorem and Representation Formula}
To state the next results we need some more terminology that we give in the next definition.
\begin{definition} \label{slice-domain}
 Let $\Omega$ be a domain in $M$.

1.   $\Omega$ is called  slice domain if $\Omega\cap\mathbb R^{p+1}\neq\emptyset$  and $\Omega_{\underline{\omega}}$ is a domain in $\mathbb{R}^{p+2}$ for every  $\underline{\omega}\in \mathbb{S}$.

2.   $\Omega$   is called  partially  symmetric with respect to $\mathbb R^{p+1}$ (p-symmetric for short) if, for   $\bx_{p}\in\R^{p+1}, r \in \mathbb R^{+},$ and $ \underline{\omega}  \in \mathbb S$,
$$\bx=\bx_p+r\underline{\omega} \in \Omega\Longrightarrow [\bx]:=\bx_p+r \mathbb S=\{\bx_p+r \underline{\omega}, \ \  \underline{\omega}\in \mathbb S\} \subseteq \Omega. $$
 \end{definition}

Denote by $\mathcal{Z}_{f}(\Omega)$  the zero set of the function $f:\Omega\subseteq M \rightarrow \mathbb{A}$. Following the proof of \cite[Theorem 3.14]{Xu-Sabadini-26},  an identity theorem for generalized partial-slice monogenic functions over slice domains can be given as follows. We refer   the reader to \cite{Xu-Sabadini-26} for the details.

\begin{theorem}  {\bf(Identity theorem)}\label{Identity-theorem}
Let $\Omega\subseteq M$ be a  slice domain and $f,g:\Omega\rightarrow \mathbb{A}$ be   generalized partial-slice monogenic functions.
If there is an imaginary $\underline{\omega}  \in \mathbb S $ such that $f=g$ on a   $(p+1)$-dimensional smooth manifold in $\Omega_{\underline{\omega}}$, then $f\equiv g$ in  $\Omega$.
\end{theorem}

The identity theorem for generalized partial-slice monogenic  functions allows to establish a representation formula. Also for this result, the proof is the same as that one given for \cite[Theorem 3.16]{Xu-Sabadini-26} and hence omitted here.  We note that, in the sequel, when the variable $\bx=\bx_p+r\underline{\omega}$ is written with the superscript, e.g. $\bx^\prime$, it denotes the pair $\bx^\prime=(\bx_p,r)$ in $\mathbb{R}^{p+2}$, to emphasise that the dependence on $\underline{\omega}\in\mathbb{S}$ is not relevant in the discussion.
\begin{theorem}  {\bf(Representation Formula)}  \label{Representation-Formula-SM}
Let $\Omega\subseteq M$ be a p-symmetric slice domain and $f:\Omega\rightarrow \mathbb{A}$ be a  generalized partial-slice monogenic function.  Then, for any $\underline{\omega}\in \mathbb{S}$ and for $\bx_p+r\underline{\omega} \in \Omega$,
\begin{equation*}\label{Representation-Formula-eq}
f(\bx_p+r \underline{\omega})=\frac{1}{2} (f(\bx_p+r\underline{\eta} )+f(\bx_p-r\underline{\eta}) )+
\frac{ 1}{2}\underline{\omega}(\underline{\eta} (  f(\bx_p-r\underline{\eta} )-f(\bx_p+r\underline{\eta}))),
\end{equation*}
for any $\underline{\eta}\in \mathbb{S}$.

Moreover, the following two functions do not depend on $\underline{\eta}$:
$$F_1(\bx')=\frac{1}{2} (f(\bx_p+r\underline{\eta} )+f(\bx_p-r\underline{\eta} ) ),$$
$$F_2(\bx')=\frac{ 1}{2}\underline{\eta}(  f(\bx_p-r\underline{\eta} )-f(\bx_p+r\underline{\eta})).$$
\end{theorem}

The interested reader can refer to \cite{Dou} for a version of the representation formula  for slice regular functions  on non-axially-symmetric domains.

\subsection{Generalized partial-slice   functions}\label{subs2.5}
Theorem   \ref{Representation-Formula-SM} shows that on some specific domains, the functions we are studying have a specific form and this fact suggests to define generalized partial-slice  (and then regular) functions, see \cite{Ghiloni-Perotti-11} for the special case of slice regular functions over real alternative $\ast$-algebras.

An open set $D$ of $\mathbb{R}^{p+2}$  is called invariant under the reflection  of the $(p+2)$-th variable if
$$ \bx'=(\bx_p,r) \in D \Longrightarrow   \bx_\diamond':=(\bx_p,-r)  \in D.$$
The  \textit{p-symmetric completion} $ \Omega_{D}$ of  $D$ is defined by
$$\Omega_{D}=\bigcup_{\underline{\omega} \in \mathbb{S}} \, \big \{\bx_p+r\underline{\omega}\in M:   \exists \bx_p \in \mathbb{R}^{p+1},  \ \exists r\geq 0,\  \mathrm{s.t.} \ (\bx_p,r)\in D \big\}.$$
\begin{definition}
A function $F=(F_1,F_2): D\longrightarrow  \mathbb{A}^{2}$ in an open set $D\subseteq  \mathbb{R}^{p+2}$, which is invariant under the reflection  of the $(p+2)$-th variable, is called  a \textit{stem function} if
the $\mathbb{A}$-valued components  $F_1, F_2$  satisfy  the so-called even-odd conditions
$$ F_1(\bx_{\diamond}')= F_1(\bx'), \qquad  F_2(\bx_{\diamond}')=-F_2(\bx'), \qquad  \bx'=(\bx_p,r) \in D.$$
Each stem function $F$ induces a (left)  generalized partial-slice
function $f=\mathcal I(F): \Omega_{D} \longrightarrow \mathbb{A}$ given by
 $$f(\bx)=\mathcal I(F)(\bx) :=F_1(\bx')+\underline{\omega} F_2(\bx'), \qquad   \bx=\bx_p+r\underline{\omega}   \in  \Omega_{D}.$$
\end{definition}
We introduce the set of  all induced  generalized partial-slice functions  on $\Omega_{D}$ denoted by
 $$ {\mathcal{GS}}(\Omega_{D},\mathbb{A}) =\Big\{f=\mathcal I(F):    \ F \ {\mbox {is an}}\ \mathbb{A}^{2} {\mbox {-valued stem function on }} D  \Big\}$$
and, for $k\in \mathbb{N}$, we set
 $${\mathcal{GS}}^{k}(\Omega_{D},\mathbb{A}) =\Big\{f=\mathcal I(F):    \ F\in C^{k}(D, \mathbb{A}^{2}) \ {\mbox {is an}}\ \mathbb{A}^{2}   {\mbox {-valued stem function}}   \Big\}.$$

\begin{definition}\label{definition-GSR}
Let $f \in {\mathcal{GS}}^{1}(\Omega_{D}, \mathbb{A})$. The function $f$ is called generalized partial-slice regular of type $(p,q)$ if its stem function $F=(F_1,F_2)\in  \mathbb{A}^{2}$ satisfies  the generalized Cauchy-Riemann equations
 \begin{eqnarray}\label{C-R}
 \left\{
\begin{array}{ll}
D_{\bx_p}  F_1- \partial_{r} F_2=0,
\\
 \overline{D}_{\bx_p}  F_2+ \partial_{r} F_1=0.
\end{array}
\right.
\end{eqnarray}
\end{definition}
As before, the type $(p,q)$ will be fixed and omitted in the sequel. Denote by $\mathcal {GSR}(\Omega_D, \mathbb{A})$ the set of all generalized partial-slice regular functions  on $\Omega_D$.

\begin{remark}
Let  $\bx'=(0,x_1,x_2,\ldots,x_p, r)=(0,x_1,x_2,\ldots,x_p, x_{m})\in \{0\} \oplus  \mathbb{R}^{m}$, where $m=p+1$, and consider   the   function $F=(F_1,F_2)\in C^1(D,\mathbb{A}^{2}) $  of the form
$$ F_1(\bx')=\sum_{i=1}^{p}f_i(\bx')v_i, \ F_2(\bx')=f_{m}(\bx'), \quad f_i(\bx')\in \mathbb{R}, i=1,2,\ldots,m.$$
Then (\ref{C-R}) reduces into the so-called Riesz system:
 \begin{eqnarray*}
 \left\{
\begin{array}{ll}
\sum_{i=1}^{m}\partial_{x_i} f_i =0,
\\
\partial_{x_i} f_j =\partial_{x_j} f_i,\quad i,j=1,2,\ldots,m,
\end{array}
\right.
\end{eqnarray*}
which was  considered in 1960 by  Stein and Weiss \cite{Stein-Weiss} in the study of  the theory of harmonic functions of several variables.
\end{remark}

Theorem  \ref{Representation-Formula-SM} provides  a relation between  functions classes $\mathcal {GSM}$ and $\mathcal {GSR}$ defined in  p-symmetric domains  can be formulated as the following result, whose proof is  the same as   \cite[Theorem 5.8]{Xu-Sabadini-26} and hence omitted here.

\begin{theorem} \label{relation-GSR-GSM}

(i) For a p-symmetric domain $\Omega=\Omega_{D}$ with $\Omega  \cap \mathbb{R}^{p+1}= \emptyset$, it holds that $\mathcal {GSM}(\Omega, \mathbb{A}) \supsetneqq \mathcal {GSR}(\Omega_{D}, \mathbb{A})$.

(ii) For a p-symmetric domain $\Omega=\Omega_{D}$ with $\Omega  \cap \mathbb{R}^{p+1}\neq \emptyset$,  it holds  that $\mathcal {GSM}(\Omega, \mathbb{A}) = \mathcal {GSR}(\Omega_{D}, \mathbb{A})$.
\end{theorem}

 The class of slice functions appears naturally as a first step in the Fueter-Sce construction in its classical formulations over quaternions, Clifford algebras, and octonions. It is then convenient to consider also for real alternative $*$-algebras the class of generalized partial-slice functions when dealing with the Fueter-Sce theorem. We point out  that the maps appearing in the construction can be factorized to produce the so-called fine structures, see \cite{CDPS-23,CDS-25}.

\subsection{Fueter polynomials}\label{secFueterpoly}
 To introduce the Fueter polynomials, we ned to introduce first the so-called Fueter variables which, in the partial-slice case, are of left and right type.
\begin{definition}\label{Fueter-variables}
 The so-called (left) Fueter variables induced by $\mathcal{B}$ are defined as
 $$ z_{\ell}(\bx)=z_{\ell}^{L}(\bx)= z_{\ell}(\bx_{p}+r\underline{\omega})= x_{\ell}+r \underline{\omega} v_\ell, \ \ell=0,1,\ldots,p.$$
Similarly, the so-called right  Fueter variables induced by $\mathcal{B}$ are defined as
 $$ z_{\ell}^{R}(\bx)= z_{\ell}^{R}(\bx_{p}+r\underline{\omega})=  x_{\ell}+r v_\ell  \underline{\omega},\ \ell=0,1,\ldots,p.$$
 \end{definition}
Let
$(a_{1}  a_{2} \cdots  a_{k})_{\otimes_{k}}$ be the product of the ordered $k$ elements $(a_1, a_2,\ldots ,a_k)\in \mathbb{A}^{k}$ in a fixed associative order $\otimes_{k}$. In particular, denote the multiplication  from left to right by
$$(a_1 a_2 \cdots a_k)_{L}:=(\cdots((a_1 a_2)a_3) \cdots )a_k,$$
and the multiplication  from  right to left by
$$(a_1 a_2 \cdots a_k)_{R}:=a_1(  \cdots (a_{k-2}(a_{k-1}a_k))\cdots ).$$

Let us recall \cite[Proposition 4.2]{HRX}, which can be used to define Fueter polynomials  in the general case of real
alternative algebras.
\begin{proposition}\label{no-order}
Let $a, a_0,a_1,  \ldots ,a_p\in \mathbb{A} $ and $(j_1,j_2,\ldots, j_k) \in \{0,1,\ldots,p\}^{k}$,  repetitions being allowed. Then the following sum is independent of the chosen associative order $\otimes_{(k+1)}$
\begin{eqnarray*}\label{sum}
  \sum_{(i_1,i_2, \ldots, i_k)\in \sigma} (a_{i_1}  a_{i_2} \cdots  a_{i_k}a)_{\otimes_{(k+1)}},
   \end{eqnarray*}
where the sum runs over  all distinguishable permutations $\sigma$ of $(j_1,j_2,\ldots, j_k)$. \\
In particular, we have
$$\sum_{(i_1,i_2, \ldots, i_k)\in \sigma} (a_{i_1}  a_{i_2} \cdots  a_{i_k}a)_L
=\sum_{(i_1,i_2, \ldots, i_k)\in \sigma} (a_{i_1}  a_{i_2} \cdots  a_{i_k}a)_{R}.$$
\end{proposition}

\begin{definition}[Fueter polynomials]\label{definition-Fueter}
For  $\mathrm{k}=(k_0,k_1,\ldots,k_{p})\in \mathbb{N}^{p+1}$, let   $\overrightarrow{\mathrm{k}}:=(j_1,j_2,\ldots, j_k)$   be an alignment with the number of $0$ in the alignment is $k_0$, the number of $1$ is $k_1$, and the number of $p$ is $k_{p}$, where $k=|\mathrm{k}|:=k_0+k_1+\cdots+k_{p}, 0\leq j_1\leq j_2\leq \ldots\leq j_k\leq p$. Define
$$\mathcal{P}_{\mathrm{k}}(\bx)  =\mathcal{P}^{L}_{\mathrm{k}}(\bx) =\frac{1}{k!}  \sum_{(i_1,i_2, \ldots, i_k)\in \sigma(\overrightarrow{\mathrm{k}})} z_{i_1}  z_{i_2} \cdots  z_{i_k},$$
where the sum runs over the $\dfrac{k!}{\mathrm{k}!}$ different permutations $\sigma(\overrightarrow{\mathrm{k}})$ of $\overrightarrow{\mathrm{k}}$. When $\mathrm{k}=(0,\ldots,0)=\mathbf{0}$, we set $\mathcal{P}_{ \mathbf{0}}(\bx)=1$; when there is at least one negative component in $\mathrm{k}$, we set $\mathcal{P}_{\mathrm{k}}(\bx)=0$.

Similarly, we can define $\mathcal{P}^{R}_{  \mathrm{k}}$ when $z_{\ell}$ are replaced by $z_{\ell}^{R}$.
\end{definition}

In fact, unlike the monogenic case in \cite[Proposition 4.4]{HRX}, the non-associativity makes almost impossible to verify directly that $\mathcal{P}_{\mathrm{k}}$  are generalized partial-slice monogenic. Later, following the methods used in \cite[Theorem 6.12]{Xu-Sabadini-26} we shall prove that the Fueter polynomials can be also obtained by the CK-extension in Definition \ref{Slice-Cauchy-Kovalevska-extension}; see Theorem \ref{V-P}. This fact implies the generalized partial-slice monogenicity of the Fueter polynomials $\mathcal{P}_{\mathrm{k}}$.

\section{CK-extensions}
In this section, we shall discuss three types of  CK-extensions over
 real alternative $\ast$-algebras.  We first consider the CK-extension of a real analytic function to a generalized partial-slice monogenic function; then we consider the extension to a monogenic function which, as we shall see, will be of special type. Finally we shall discuss the CK-extension of a pair of real analytic functions to harmonic functions which will be, again, of special type.
\subsection{CK-extension for generalized partial-slice monogenic functions}
 First, we define  a  generalized partial-slice monogenic  Cauchy-Kovalevskaya extension   starting from  real analytic functions defined in some domain in $\mathbb{R}^{p+1}$.
\begin{definition}[CK-extension]\label{Slice-Cauchy-Kovalevska-extension}
Let   $\Omega_{0}$ be a domain in  $\mathbb{R}^{p+1}$ and denote $\Omega_{0}^{\ast}$   by
$$\Omega_{0}^{\ast}=\{ \bx_{p}+\underline{\bx}_{q}:\bx_{p}\in \Omega_{0}, \underline{\bx}_{q}\in \mathbb{R}^{q}\}.$$
Given a  real analytic function  $f_{0}\in C^{\omega}(\Omega_0, \mathbb{A})$, assume that there exists a left generalized partial-slice monogenic function $f^{\ast}$ in some domain $\Omega \subseteq \Omega_{0}^{\ast}$ with $\Omega_{0} \subset \Omega$ and such that $f^{\ast}(\bx_{p})=f_0(\bx_{p})$.
Such  a function $f^{\ast}$  is called the left generalized partial-slice Cauchy-Kovalevskaya extension (CK-extension for short) of  $f_{0}$.
\end{definition}

In the sequel, we shall use a fundamental  result whose proof     follows from   \cite[Lemma 3.15]{Xu-Sabadini-26}.

\begin{lemma}\label{a-w-b}
Let $a\in\mathbb{ R}^{p+1}, \underline{\omega}\in \mathbb{S}$. Then it holds that for any $b\in \mathbb{A}$
$$a(\underline{\omega}b)=\underline{\omega}(a^{c}b).$$
\end{lemma}
\begin{proof}
Note that, for any $a,b,\underline{\omega}\in \mathbb{A}$,
$$[a,  \underline{\omega}, b]+[a^{c},\underline{\omega}, b]= [a+a^{c},  \underline{\omega},b].$$
In the particular case where  $a\in M$ and $b,\underline{\omega}\in \mathbb{A}$,
$$[a,  \underline{\omega}, b]+[a^{c},\underline{\omega}, b]=0,$$
which gives
$$ [a,  \underline{\omega}, b]=-[a^{c},\underline{\omega}, b]=[ \underline{\omega},a^{c}, b],$$
where we used the fact that the algebra is alternative to get the second equality.
Combining this with the fact that
$$ a \underline{\omega} =\underline{\omega} a^{c}, \quad a\in\mathbb{ R}^{p+1}, \underline{\omega}\in \mathbb{S}, $$
we get the assertion.
\end{proof}

\begin{theorem}\label{slice-Cauchy-Kovalevsky-extension}
Let   $\Omega_{0}$ be a domain in  $\mathbb{R}^{p+1}$ and  consider the real analytic function $f_{0}\in C^{\omega}(\Omega_0, \mathbb{A})$.  Then the function given by
\begin{eqnarray}\label{series-CK}
 CK[f_{0}](\bx):=\sum_{k=0}^{+\infty} \frac{r^{2k}}{(2k)!}  (-\Delta_{\bx_p})^{k} f_{0}(\bx_p)
+\underline{\omega}  \sum_{k=0}^{+\infty} \frac{r^{2k+1}}{(2k+1)!}(-\Delta_{\bx_p}) ^{k}(D_{\bx_p} f_{0}(\bx_p)),
\end{eqnarray}
is a   CK-extension of $f_0$ in a  p-symmetric slice domain $\Omega \subseteq \Omega_{0}^{\ast}$ with $\Omega_{0} \subset \Omega$, where $\bx=\bx_p+ r \underline{\omega}$ with $\bx_{p}\in\R^{p+1}, r \geq 0,$ and $\underline{\omega}  \in \mathbb S$.
\end{theorem}
\begin{proof}
First of all, it should be pointed that  the series in (\ref{series-CK})  does not depend  on $\underline{\omega}$ at any $\bx_p\in\Omega_{0} \subseteq \mathbb{R}^{p+1}$. Let $f_{0}\in C^{\omega}(\Omega_0, \mathbb{A})$.
We show that $CK[f_{0}]$ given by the series in (\ref{series-CK}) is well-defined  on some  p-symmetric slice domain.
To this end, we recall a well-known result: a real-valued function $g$ is real analytic in $\Omega_{0}$ if and only if, for any  compact set $K\subset \Omega_0$, there exists some constant $C=C_{g, K}>0$ such that
$$| \partial^{\mathrm{k}} g(\bx_{p}) |\leq \mathrm{k}! C^{1+|\mathrm{k}|},  \quad  \forall\bx_{p}\in K, \mathrm{k} \in \mathbb{N}^{p+1},  $$
where  $\partial^{\mathrm{k}} =\frac{\partial^{k}  }{\partial_{x_{0}}^{k_0}\partial_{x_{1}}^{k_1}\cdots\partial_{x_{p}}^{k_p} }$ with $\mathrm{k}=(k_0,k_1,\ldots,k_{p})$ and $k=|\mathrm{k}|$.\\
Hence, given $f_{0} \in C^{\omega}(\Omega_0, \mathbb{A})$, there exist  some constants $\lambda=\lambda_{p,K}, C=C_{K}>0$ such that
$$| \Delta_{\bx_p}^{k} f_{0}(\bx_{p}) |\leq \lambda (2k)! C^{2k},  \quad  \bx_{p}\in K,  k\in    \mathbb{N},  $$
which gives that, for all $\bx_{p}\in K$,
$$\Big|\sum_{k=0}^{+\infty} \frac{r^{2k}}{(2k)!}  (-\Delta_{\bx_p})^{k} f_{0}(\bx_p)\Big|\leq \lambda \sum_{k=0}^{+\infty} (rC)^{2k},$$
and
$$\Big| \sum_{k=0}^{+\infty} \frac{r^{2k+1}}{(2k+1)!}(-\Delta_{\bx_p}) ^{k} f_{0}(\bx_p)\Big| \leq  \lambda r \sum_{k=0}^{+\infty}  \frac{(rC)^{2k}}{2k+1}.$$
Hence, the right side of    (\ref{series-CK})  converges      normally in  the p-symmetric, slice domain
$$\Omega=\bigcup_{K\subset  \Omega_{0}} \mathring{K} \times \{\underline{\bx}_{q}\in \mathbb{R}^{q}: |\underline{\bx}_{q}|<\frac{1}{C_{K}}\} \subseteq \Omega_{0}^{\ast}.$$

Now it remains to show that $CK[f_{0}]\in \mathcal{GSM}^{L}(\Omega, \mathbb{A})$. To this end, we fix $\underline{\omega}\in \mathbb{S}$ and we prove that, for $\bx_p+ r \underline{\omega}\in \Omega_{\underline{\omega}}$,
  \begin{eqnarray}\label{w-w-D-f}
  D_{\bx_p} (\underline{\omega}  (D_{\bx_p} f_{0}(\bx_p)))=\underline{\omega}\Delta_{\bx_p} f_{0}(\bx_p), \end{eqnarray}
and
  \begin{eqnarray}\label{w-w-f}
 \underline{\omega} (\underline{\omega} (D_{\bx_p} f_{0}) )= -D_{\bx_p} f_{0}.
 \end{eqnarray}
We note that (\ref{w-w-f}) follows by the Artin's theorem
 $$ \underline{\omega} (\underline{\omega} (D_{\bx_p} f_{0}) )= (\underline{\omega}^{2})(D_{\bx_p} f_{0}) =-D_{\bx_p} f_{0}.$$
To prove (\ref{w-w-D-f}), we recall that $M$ is a hypercomplex subspace of the alternative $\ast$-algebra $\mathbb A$, so that for all $i,j=0,1,\ldots,p,$ with $ i<j,$ and for all $a\in \mathbb A$, we have
 $$ v_i^{c}   v_j+v_j^{c}   v_i =0, \ \ \ \ [v_i^{c},v_j, a] +[v_j^{c},v_i,a]=0,$$
which give\begin{eqnarray*}
&&\sum_{i,j=0,  i<j}^{p}  \big(  v_i^{c}  (v_j  a) +  v_j^{c}  (v_i  a) \big)
 \\
 &=&  \sum_{i,j=0,  i<j}^{p} \big( (v_i^{c}   v_j)  a-[v_i^{c},v_j, a]+ (v_j^{c}   v_i)  a-[v_j^{c},v_i, a] \big)
 \\
 &=& \sum_{i,j=0,  i<j}^{p}  (v_i^{c}   v_j+v_j^{c}   v_i)  a-\sum_{i,j=0,  i<j}^{p} ([v_i^{c},v_j, a] +[v_j^{c},v_i, a]  )
 \\
  &=&0.
\end{eqnarray*}
Hence, by Lemma  \ref{a-w-b} and Proposition \ref{artin} and recalling that $f_{0}\in C^{\omega}(\Omega_0, \mathbb{A})$,
\begin{eqnarray*}
 &&D_{\bx_p} (\underline{\omega}  (D_{\bx_p} f_{0}))
=\sum_{i=0}^{p}v_i (    \underline{\omega}  (D_{\bx_p} \partial_{x_i}f_{0}))
 \\
 &=&\sum_{i=0}^{p} \underline{\omega} (v_i^{c}  (D_{\bx_p} \partial_{x_i}f_{0}))
 = \sum_{i,j=0}^{p} \underline{\omega} (v_i^{c}  (v_j  \partial_{x_i}\partial_{x_j}f_{0}))
 \\
  &=& \sum_{i =0 }^{p} \underline{\omega} (v_i^{c}  (v_i  \partial_{x_i}^{2} f_{0}))  +
  \sum_{ 0\leq i<j\leq p} \underline{\omega}\big(  v_i^{c}  (v_j  \partial_{x_i}\partial_{x_j}f_{0}) +  v_j^{c}  (v_i \partial_{x_i}\partial_{x_j}f_{0}) \big)
 \\
 &=& \sum_{i =0 }^{p} \underline{\omega}( (v_i^{c}  v_i ) \partial_{x_i}^{2} f_{0})  )
 =\underline{\omega} \Delta_{\bx_p} f_{0}.
\end{eqnarray*}
Now we can use (\ref{w-w-D-f}) and (\ref{w-w-f}) to obtain
\begin{eqnarray*}
 & &D_{\underline{\omega}} CK[f_{0}](\bx_p+ r \underline{\omega})\\
 &=& (D_{\bx_p}+\underline{\omega}\partial_{r}) \Big(\sum_{k=0}^{+\infty} \frac{r^{2k}}{(2k)!}  (-\Delta_{\bx_p})^{k} f_{0}(\bx_p)\Big)\\
&& +(D_{\bx_p}+\underline{\omega}\partial_{r}) \Big( \underline{\omega}  \sum_{k=0}^{+\infty} \frac{r^{2k+1}}{(2k+1)!}(-\Delta_{\bx_p}) ^{k}(D_{\bx_p} f_{0}(\bx_p)) \Big)
 \\
&=&  \sum_{k=0}^{+\infty} \frac{r^{2k}}{(2k)!}  (-\Delta_{\bx_p})^{k}D_{\bx_p} f_{0}(\bx_p)+\underline{\omega}\sum_{k=1}^{+\infty} \frac{r^{2k-1}}{(2k-1)!}  (-\Delta_{\bx_p})^{k} f_{0}(\bx_p)
\\
& &+\underline{\omega}  \sum_{k=0}^{+\infty} \frac{r^{2k+1}}{(2k+1)!}(-\Delta_{\bx_p}) ^{k}  \Delta_{\bx_p} f_{0}(\bx_p)-
\sum_{k=0}^{+\infty} \frac{r^{2k}}{(2k)!}(-\Delta_{\bx_p}) ^{k}D_{\bx_p} f_{0}(\bx_p)
\\
&=&0,
\end{eqnarray*}
so that $CK[f_{0}]\in \mathcal{GSM}^{L}(\Omega, \mathbb{A})$, as desired.
\end{proof}
\begin{remark}
Note that the series in (\ref{series-CK}) can be rewritten as
\begin{eqnarray*}
 CK[f_{0}](\bx)=\sum_{k=0}^{+\infty} \frac{\underline{\bx}_q^{2k}}{(2k)!}   \Delta_{\bx_p}^{k} f_{0}(\bx_p)
+  \sum_{k=0}^{+\infty} \frac{\underline{\bx}_q^{2k+1}}{(2k+1)!}\Delta_{\bx_p}^{k}(D_{\bx_p} f_{0}(\bx_p)),
\end{eqnarray*}
where $\bx=\bx_p+ \underline{\bx}_q$.
Similarly, we can prove that the series
\begin{eqnarray*}
 CK^{R}[f_{0}](\bx)=\sum_{k=0}^{+\infty} \frac{\underline{\bx}_q^{2k}}{(2k)!}   \Delta_{\bx_p}^{k} f_{0}(\bx_p)
+  \sum_{k=0}^{+\infty} \Delta_{\bx_p}^{k}( f_{0}(\bx_p)D_{\bx_p})\frac{\underline{\bx}_q^{2k+1}}{(2k+1)!},
\end{eqnarray*}
is a right  generalized partial-slice CK-extension   of  $f_{0}$.
\end{remark}

From  the obvious fact  $CK[f_{0}](\bx_p)=f_{0}(\bx_p)$ and Theorem \ref{Identity-theorem},  we get
\begin{theorem}\label{CK-slice-monogenic}
Let $f_{0}:  \mathbb{R}^{p+1} \to \mathbb{A}$ be a  polynomial. Then $CK[f_0]$  is the unique extension  of $f_{0}$ to $M$ which is generalized partial-slice monogenic.
\end{theorem}

\begin{definition}\label{Cauchy-Kovalevska-extension-Fueter}
For  $\mathrm{k}=(k_0,k_1,\ldots,k_{p})\in\mathbb{N}^{p+1}$ and $\bx_{p}^{\mathrm{k}}=x_0^{k_0}x_1^{k_1}\ldots x_p^{k_p}$, define
  $$V_{\mathrm{k}}(\bx)=\frac{1}{\mathrm{k}!}CK[\bx_{p}^{\mathrm{k}}](\bx),$$
 where $\mathrm{k}!=k_0!k_1!\cdots k_p!$.
\end{definition}
We now note that, by definition, we have
$$z_{\ell}(\bx)=  CK[x_{\ell}](\bx)=x_\ell+r\underline{\omega} v_\ell,  \quad \ell=0,1,\ldots,p,$$
  where $\bx=\bx_p+ r \underline{\omega}$ with $\bx_{p}\in\R^{p+1}, r \geq 0,$ and $\underline{\omega}  \in \mathbb S$.
Thus, using
 the same arguments to prove \cite[Theorem 6.12]{Xu-Sabadini-26}, we deduce the following:
\begin{theorem}\label{V-P}
 For each $\mathrm{k} \in \mathbb{N}^{p+1}$, there holds
 $ V_{\mathrm{k}}(\bx)= \mathcal{P}_{\mathrm{k}}(\bx),$ so that the polynomial $\mathcal{P}_{\mathrm{k}}$ is generalized partial-slice monogenic.
  \end{theorem}

\subsection{Generalized CK-extension for monogenic  functions}
In this subsection, we shall establish   the generalized CK-extension for monogenic functions over
 real alternative $\ast$-algebras.
\begin{lemma} \label{CK-formular-Four}
For any function $g=g(\bx_p)\in C^{1}(\mathbb{R}^{p+1}, \mathbb{A})$ and  $k\in \mathbb{N}$, the following formulas  hold
\begin{equation}\label{CK-formular-1} D_{\bx_p} (\underline{\bx}_q^{2k}g) = \underline{\bx}_q^{2k}(D_{\bx_p}g), \end{equation}
\begin{equation}\label{CK-formular-2} D_{\bx_p} (\underline{\bx}_q^{2k+1}g)  = \underline{\bx}_q^{2k+1}(\overline{D}_{\bx_p}g),\end{equation}
\begin{equation}\label{CK-formular-3}
D_{\underline{\bx}_q} (\underline{\bx}_q^{2k}g) =
 -2k \underline{\bx}_q^{2k-1}g, \end{equation}
\begin{equation}\label{CK-formular-4} D_{\underline{\bx}_q}(\underline{\bx}_q^{2k+1}g)=-(2k+q) \underline{\bx}_q^{2k}g. \end{equation}
\end{lemma}
\begin{proof}
Let us first show (\ref{CK-formular-1}) and  (\ref{CK-formular-2}). Set $r=|\underline{\bx}_q|=(x_{p+1}^2+\ldots+x_{p+q}^2)^{1/2}$. Noting  that  $\underline{\bx}_q^{2k}=(-1)^{k}r^{2k}$ is real and does not depend on $x_i, i=0,1,\ldots,p$,  for any function $g=g(\bx_p)\in C^{1}(\mathbb{R}^{p+1}, \mathbb{A})$ and  $k\in \mathbb{N}$ it follows that
$$D_{\bx_p} (\underline{\bx}_q^{2k}g) = \underline{\bx}_q^{2k}(D_{\bx_p}g),  $$
and
$$D_{\bx_p} (\underline{\bx}_q^{2k+1}g)  =(\underline{\bx}_q^{2k}) D_{\bx_p} (\underline{\bx}_q g).$$
 Hence, by Lemma \ref{a-w-b}, we have
$$D_{\bx_p} (\underline{\bx}_q^{2k+1}g)  =(\underline{\bx}_q^{2k}) D_{\bx_p} (\underline{\bx}_q g)= \underline{\bx}_q^{2k+1}(\overline{D}_{\bx_p}g).$$
\\
Now it remains to prove (\ref{CK-formular-3}) and (\ref{CK-formular-4}). Direct calculations show that
$$D_{\underline{\bx}_q} (r^{2k}g) =  \sum_{i=p+1}^{p+q} (e_i \partial_{x_i} r^{2k})g =2k(r^{2k-2}\underline{\bx}_q)g.$$
\\
Noticing  also that $\underline{\bx}_q^{2k}=(-1)^{k}r^{2k},$ we get
$$D_{\underline{\bx}_q} (r^{2k}g) =(-1)^{k-1}2k(  \underline{\bx}_q^{2k-1})g= (-1)^{k-1} 2 k  \underline{\bx}_q^{2k-1} g.$$
Hence,
$$D_{\underline{\bx}_q} (\underline{\bx}_q^{2k}g) = (-1)^{k} D_{\underline{\bx}_q} (r^{2k}g)=
 -2k \underline{\bx}_q^{2k-1}g,$$
which gives (\ref{CK-formular-3}).

Then, we observe that
$$[D_{\underline{\bx}_q},\underline{\bx}_q, g ]=\sum_{i=p+1}^{p+q}\sum_{j=p+1}^{p+q}[e_i, e_j \partial_{x_i} x_j,g]=\sum_{i=p+1}^{p+q}\sum_{j=p+1}^{p+q}[e_i, e_j  ,g] \delta_{i,j}=0, $$
and we deduce
$$D_{\underline{\bx}_q}( \underline{\bx}_qg )= ( D_{\underline{\bx}_q} \underline{\bx}_q ) g=-qg. $$
Hence, (\ref{CK-formular-4}) can be proved as
\begin{eqnarray*}
  D_{\underline{\bx}_q}(\underline{\bx}_q^{2k+1}g)
 &=&D_{\underline{\bx}_q}(\underline{\bx}_q^{2k} (\underline{\bx}_qg))
 \\
  &=&(-1)^{k}\big((D_{\underline{\bx}_q} r^{2k})(\underline{\bx}_qg)+ r^{2k} D_{\underline{\bx}_q} (\underline{\bx}_qg) \big )
 \\
  &=&(-1)^{k}\big((  2kr^{2k-2} \underline{\bx}_q )(\underline{\bx}_qg)-q r^{2k}  g \big )
 \\
 &=&(-1)^{k}\big((  2kr^{2k-2}( \underline{\bx}_q  \underline{\bx}_q) g-q r^{2k}  g \big )
 \\
&=&(-1)^{k+1} (2k+q) r^{2k} g
\\
&=&-(2k+q) \underline{\bx}_q^{2k}g,
\end{eqnarray*}
where the third and forth equalities  follow from  the Artin's theorem. The proof is complete.
\end{proof}
\begin{lemma} \label{CK-formular-Four-D}
For any function $g=g(\bx_p)\in C^{1}(\mathbb{R}^{p+1}, \mathbb{A})$ and  $k\in \mathbb{N}$, the following formulas hold
\begin{equation}\label{CK-formular-1-D}
D_{\bx} (\underline{\bx}_q^{2k}g) = \underline{\bx}_q^{2k}(D_{\bx_p}g)-2k \underline{\bx}_q^{2k-1}g, \end{equation}
\begin{equation}\label{CK-formular-2-D}
 D_{\bx} (\underline{\bx}_q^{2k+1}g)  = \underline{\bx}_q^{2k+1}(\overline{D}_{\bx_p}g)-(2k+q) \underline{\bx}_q^{2k}g,
 \end{equation}
 \begin{equation}\label{CK-formular-1-D-bar}
 \overline{D}_{\bx} (\underline{\bx}_q^{2k}g) = \underline{\bx}_q^{2k}(\overline{D}_{\bx_p}g)+2k \underline{\bx}_q^{2k-1}g, \end{equation}
\begin{equation}\label{CK-formular-2-D-bar}
 \overline{D}_{\bx} (\underline{\bx}_q^{2k+1}g)  = \underline{\bx}_q^{2k+1}(D_{\bx_p}g)+(2k+q) \underline{\bx}_q^{2k}g.
 \end{equation}
\end{lemma}

\begin{proof}
From   (\ref{CK-formular-1}) and  (\ref{CK-formular-3}) we get directly (\ref{CK-formular-1-D}), while (\ref{CK-formular-2-D}) follows from   (\ref{CK-formular-2}) and  (\ref{CK-formular-4}). Similarly, we can prove (\ref{CK-formular-1-D-bar}) and (\ref{CK-formular-2-D-bar}).
\end{proof}

\begin{theorem}[Generalized CK-extension]\label{generalized-CK-extension}
Let  $\Omega_{0}\subseteq \mathbb{R}^{p+1}$ be a domain and  consider the real analytic function   $A_{0} \in C^{\omega}(\Omega_0, \mathbb{A})$. Then there exists a unique  sequence of  real analytic functions $\{A_{k}\}_{k=1}^{\infty} \subset C^{\omega}(\Omega_0, \mathbb{A}) $ such that the series
\begin{equation}\label{Cauchy-Kovalevska-extension-sum}
f(\bx)  = \sum_{k=0}^{+\infty} \underline{\bx}_{q}^{k}A_{k}(\bx_p) \end{equation}
is convergent in a p-symmetric  slice domain $\Omega \subseteq \Omega_{0}^{\ast}$ with $\Omega_{0} \subset \Omega$ and its sum $f$ is monogenic in $\Omega$.
Furthermore,  the function $f$ can be written formally as
\begin{equation}\label{Cauchy-Kovalevska-extension-expression}
 f(\bx) =\Gamma(\frac{q}{2})\Big( \frac{r\sqrt{\Delta_{\bx_{p}}}}{2}\Big)^{-\frac{q}{2}} \Big[ \frac{r\sqrt{\Delta_{\bx_{p}}}}{2} J_{\frac{q}{2}-1}(r\sqrt{\Delta_{\bx_{p}}})[A_{0}(\bx_{p})] + \frac{\underline{\bx}_{q}}{2} J_{\frac{q}{2}}(r\sqrt{\Delta_{\bx_{p}}}) [D_{\bx_{p}} A_{0}(\bx_{p})] \Big] , \end{equation}
where $r=|\underline{\bx}_{q}|$, $\sqrt{\Delta_{\bx_{p}}}$ denotes the square root of the Laplacian $\Delta_{\bx_{p}}$ (of which only even powers
occur in the resulting series),  and $J_{n}$ is the Bessel function of the first kind of order $n$ given by
$$J_{n}(x)= \sum_{k=0}^{+\infty}\frac{(-1)^{k}(x/2)^{n+2k}}{k!\Gamma(n+k+1)}.$$
\end{theorem}
\begin{proof}
First, let us show  the uniqueness, i.e.   if the series in (\ref{Cauchy-Kovalevska-extension-sum}) exists  and is convergent, then it is unique and given by (\ref{Cauchy-Kovalevska-extension-expression}). To this end,  applying the operator $D_{\bx}$ to the  sum $f(\bx)$ in (\ref{Cauchy-Kovalevska-extension-sum}), we have by Lemma  \ref{CK-formular-Four-D}
 \begin{eqnarray}\label{D-f}
D_{\bx} f(\bx)&=&
   \sum_{k=0}^{+\infty}D_{\bx} (\underline{\bx}_{q}^{2k}A_{2k}(\bx_p)  )+  D_{\bx} (\underline{\bx}_{q}^{2k+1}A_{2k+1}(\bx_p) ) \notag
 \\
  &=&\sum_{k=0}^{+\infty} \Big(\underline{\bx}_{q}^{2k} (D_{\bx_p}A_{2k}(\bx_p))- 2k\underline{\bx}_{q}^{2k-1}A_{2k}(\bx_p)\notag
 \\
  & &\quad + \underline{\bx}_{q}^{2k+1} ( \overline{D}_{\bx_p}A_{2k+1}(\bx_p) )-(2k+q) \underline{\bx}_{q}^{2k} A_{2k+1}(\bx_p)   \Big)\notag\\
  &=&\sum_{k=0}^{+\infty} \Big(\underline{\bx}_{q}^{2k} (D_{\bx_p}A_{2k}(\bx_p)-(2k+q) A_{2k+1}(\bx_p))\notag\\
  &&\quad +
  \underline{\bx}_{q}^{2k+1} ( \overline{D}_{\bx_p}A_{2k+1}(\bx_p) -(2k+2) A_{2k+2}(\bx_p) )   \Big).
\end{eqnarray}
Hence,  if $f$ is monogenic, then for all $k\in \mathbb{N}$,
\begin{eqnarray}\label{inter-relation-f}
 \left\{
\begin{array}{ll}
(2k+q) A_{2k+1}(\bx_{p})=D_{\bx_p} A_{2k}(\bx_{p}),
\\
(2k+2) A_{2k+2}(\bx_{p})=\overline{D}_{\bx_p} A_{2k+1}(\bx_{p}),
\end{array}
\right.
\end{eqnarray}
which implies that
\begin{eqnarray*}
A_{2k}(\bx_{p})& =&\frac{1}{2k}\overline{D}_{\bx_p} A_{2k-1}(\bx_{p})\\
 &=&\frac{1}{2k(2k+q-2)}   \overline{D}_{\bx_p} (D_{\bx_p}A_{2k-2}(\bx_{p}))
 \\
&=&\frac{1}{2k(2k+q-2)}   \Delta_{\bx_p} A_{2k-2}(\bx_{p}),\end{eqnarray*}
where the last equality follows from  Lemma \ref{d-dbar-laplace}.  \\
Using repeatedly this formula, we get
$$A_{2k}(\bx_{p})=\frac{ \Delta_{\bx_p}^{k} A_{0}(\bx_{p})}{(2k)!!(2k+q-2)(2k+q-4)\cdots q}
=\frac{\Gamma(\frac{q}{2})}{2^{2k}k!\Gamma(k+\frac{q}{2})}   \Delta_{\bx_p}^{k} A_{0}(\bx_{p}).$$
Similarly,    it follows from (\ref{inter-relation-f})  and Lemma \ref{d-dbar-laplace} that
\begin{eqnarray*}
A_{2k+1}(\bx_{p})& =&\frac{1}{2k+q} D_{\bx_p} A_{2k}(\bx_{p})\\
 &=& \frac{1}{2k(2k+q)}  D_{\bx_p}( \overline{D}_{\bx_p} A_{2k-1}(\bx_{p}))
 \\
&=&\frac{1}{2k(2k+q)}  \Delta_{\bx_p} A_{2k-1}(\bx_{p})\\
&=&\cdots
\\
&=&\frac{1}{(2k)!!(2k+q)(2k+q-2)\cdots (q+2)}   \Delta_{\bx_p}^{k} A_{1}(\bx_{p}),
\\
&=&\frac{1}{(2k)!!(2k+q)(2k+q-2)\cdots (q+2)q}   \Delta_{\bx_p}^{k} D_{\bx_p} A_{0}(\bx_{p})
\\
&=&\frac{\Gamma(\frac{q}{2})}{2^{2k+1}k!\Gamma(k+\frac{q}{2}+1)} \Delta_{\bx_p}^{k} D_{\bx_p}A_{0}(\bx_{p}).
\end{eqnarray*}
Consequently,  the expression of $f$ is given by
\begin{equation}\label{Cauchy-Kovalevska-extension-sum-1-1}
f(\bx)=\sum_{k=0}^{+\infty} \frac{\Gamma(\frac{q}{2})(-1)^{k}r^{2k}}{2^{2k}k!\Gamma(k+\frac{q}{2})}   \Delta_{\bx_p}^{k} A_{0}(\bx_{p})+
\underline{\bx}_q  \sum_{k=0}^{+\infty} \frac{\Gamma(\frac{q}{2})(-1)^{k}r^{2k}}{2^{2k+1}k!\Gamma(k+\frac{q}{2}+1)} \Delta_{\bx_p}^{k} (D_{\bx_p}A_{0}(\bx_{p}))\end{equation}
 and thus $f$ takes the form (\ref{Cauchy-Kovalevska-extension-expression}) and so is unique. Finally, following the classical method as in \cite[Theorem 14.8]{Brackx}, we can  show   that the two series in (\ref{Cauchy-Kovalevska-extension-sum-1-1}) both converge  in a  p-symmetric slice domain $\Omega \subseteq \Omega_{0}^{\ast}$ with $\Omega_{0} \subset \Omega$. Here we only consider the first series in (\ref{Cauchy-Kovalevska-extension-sum-1-1}), the convergence of the other can be obtained similarly.  From the proof of Theorem \ref{slice-Cauchy-Kovalevsky-extension},
given $A_{0} \in  C^{\omega}(\Omega_{0}, \mathbb{A})$, there exist  some constants $\lambda=\lambda_{p,K}, C=C_{K}>0$ such that
$$| \Delta_{\bx_p}^{k} A_{0}(\bx_{p}) |\leq \lambda (2k)! C^{2k},  \quad   \bx_{p}\in K,  k\in    \mathbb{N}.  $$
which implies that, for all $\bx_{p}\in K$,
$$\Big|\sum_{k=0}^{+\infty} \frac{  (-1)^{k}r^{2k}}{2^{2k}k!\Gamma(k+\frac{q}{2})}   \Delta_{\bx_p}^{k} A_{0}(\bx_{p})\Big|
\leq \lambda  \sum_{k=0}^{+\infty} \frac{  r^{2k}(2k)! C^{2k} }{2^{2k}k!\Gamma(k+\frac{q}{2})}.   $$
Observing that   the series on the right side of the above inequality converges for $r<C^{-1}$, we infer that
the series $$\sum_{k=0}^{+\infty} \frac{\Gamma(\frac{q}{2})(-1)^{k}r^{2k}}{2^{2k}k!\Gamma(k+\frac{q}{2})}   \Delta_{\bx_p}^{k} A_{0}(\bx_{p})$$ converges  normally in the  p-symmetric slice domain
 $$\Omega=\bigcup_{K\subset  \Omega_{0}} \mathring{K} \times \{\underline{\bx}_{q}\in \mathbb{R}^{q}: |\underline{\bx}_{q}|<\frac{1}{C_{K}}\}.$$
The proof is complete.
\end{proof}

Theorem \ref{generalized-CK-extension} allows to introduce the following:
\begin{definition}
 The function given by the series in  (\ref{Cauchy-Kovalevska-extension-sum}) is called  monogenic  generalized CK-extension of  $A_0$ and is  denoted by $GCK[A_{0}]$.
 \end{definition}
 Note that $GCK[A_{0}]=CK[A_{0}]$ when $q=1$.
\begin{remark}\label{axial}
Setting $\underline{\bx}_{q}= \underline{\omega}r$, $r\geq 0$, it is clear that the function in
\eqref{Cauchy-Kovalevska-extension-sum} can be rewritten as
$$f(\bx)  = \sum_{k=0}^{+\infty}\big( (-1)^kr^{2k}A_{2k}(\bx_p)+\underline{\omega}(-1)^kr^{2k+1}A_{2k+1}(\bx_p)\big)=F_1(\bx')+\underline{\omega}F_2(\bx'),$$
 where $\bx'=(\bx_p,r)$. It is then immediate that $F_1(\bx_\diamond')=F_1(\bx')$ and $F_2(\bx_\diamond')=-F_2(\bx')$ and thus $f$ is a generalized partial-slice function. Thus $f$ in
\eqref{Cauchy-Kovalevska-extension-sum} belongs to the set $\mathcal{M}(\Omega,\mathbb{A})\cap \mathcal{GS}(\Omega,\mathbb{A})$ that later  shall be denoted by $\mathcal{AM}(\Omega,\mathbb{A})$. When $\mathbb{A}=\mathbb{R}_{0,m}$ with $M=\mathbb{R}^{m+1}$, it is the well-known set of axially monogenic functions, which is the range of the Fueter-Sce map.

\end{remark}

\subsection{Generalized CK-extension for harmonic functions}
In this subsection, we shall establish   the generalized CK-extension for harmonic functions over
 real alternative $\ast$-algebras.
\begin{theorem}[Harmonic generalized CK-extension]\label{Harm-generalized-CK-extension}
Let  $\Omega_{0}\subseteq \mathbb{R}^{p+1}$ be a domain and  consider two given real analytic  functions $A_{k} \in C^{\omega}(\Omega_0, \mathbb{A})$, $k=0,1$. Then there exists a unique sequence of real analytic functions $\{A_{k}\}_{k=2}^{\infty}\subset C^{\omega}(\Omega_0, \mathbb{A}) $ such that the series
\begin{equation}\label{Harm-Cauchy-Kovalevska-extension-sum}
f(\bx)  =  \sum_{k=0}^{+\infty} \underline{\bx}_{q}^{k}A_{k}(\bx_p) \end{equation}
is convergent in a p-symmetric  slice domain $\Omega \subseteq \Omega_{0}^{\ast}$ with $\Omega_{0} \subset \Omega$ and its sum $f$ is harmonic  in $\Omega$.
Furthermore,  the function $f$ can be written in a formal way as
\begin{equation}\label{Harm-Cauchy-Kovalevska-extension-expression}
 f(\bx) =\Gamma(\frac{q}{2})\Big( \frac{r\sqrt{\Delta_{\bx_{p}}}}{2}\Big)^{-\frac{q}{2}} \Big[ \frac{r\sqrt{\Delta_{\bx_{p}}}}{2} J_{\frac{q}{2}-1}(r\sqrt{\Delta_{\bx_{p}}})[A_{0}(\bx_{p})] + \frac{q\underline{\bx}_{q} }{2} J_{\frac{q}{2}}(r\sqrt{\Delta_{\bx_{p}}}) [A_{1}(\bx_{p})] \Big] , \end{equation}
and satisfies  the initial condition
\begin{eqnarray}\label{initial-conditions-f}
 \left\{
\begin{array}{ll}
f(\bx)\mid_{\underline{\bx}_{q}=0}= A_{0}(\bx_{p}),
\\
 D_{\underline{\bx}_q}f(\bx)\mid_{\underline{\bx}_{q}=0}=-q A_{1}(\bx_{p}).
\end{array}
\right.
\end{eqnarray}
\end{theorem}
\begin{proof}
  Given the   initial real analytic  functions  $A_{k} \in C^{\omega}(\Omega_0, \mathbb{A})$, $k=0,1$,  we shall prove the existence and uniqueness of the harmonic extension (\ref{Harm-Cauchy-Kovalevska-extension-expression}). The existence of such a function in  (\ref{Harm-Cauchy-Kovalevska-extension-expression}) can be obtained  reasoning as in the proof of Theorem \ref{generalized-CK-extension}.   The uniqueness can be shown as follows.  By applying the Laplacian $\Delta_{\bx}$ to (\ref{Harm-Cauchy-Kovalevska-extension-sum}) we have, by Lemma \ref{d-dbar-laplace}, (\ref{D-f}),  and Lemma \ref{CK-formular-Four}, that
   \begin{eqnarray*}
\Delta_{\bx} f(\bx)&=&\overline{D}_{\bx} (D_{\bx} f(\bx))
 \\
  &=&\sum_{k=0}^{+\infty} \Big(\overline{D}_{\bx}  (\underline{\bx}_{q}^{2k} (D_{\bx_p}A_{2k}(\bx_p)-(2k+q) A_{2k+1}(\bx_p)))\\
  &&\quad +
  \overline{D}_{\bx} (\underline{\bx}_{q}^{2k+1} ( \overline{D}_{\bx_p}A_{2k+1}(\bx_p) -(2k+2) A_{2k+2}(\bx_p) )  ) \Big)
 \\
  &=&\sum_{k=0}^{+\infty}  \Big( \underline{\bx}_{q}^{2k}( \overline{D}_{\bx_p}  (D_{\bx_p}A_{2k}(\bx_p)-(2k+q) A_{2k+1}(\bx_p)))\\
  &&\quad + 2k \underline{\bx}_{q}^{2k-1} (D_{\bx_p}A_{2k}(\bx_p)-(2k+q) A_{2k+1}(\bx_p))
  \\
  &&\quad+
  \underline{\bx}_{q}^{2k+1} ( D_{\bx_p} (  \overline{D}_{\bx_p}A_{2k+1}(\bx_p) -(2k+2) A_{2k+2}(\bx_p) ))
  \\
  && \quad + (2k+q)   \underline{\bx}_{q}^{2k}  ( \overline{D}_{\bx_p}A_{2k+1}(\bx_p) -(2k+2) A_{2k+2}(\bx_p) ) \Big)
  \\
  &=&\sum_{k=0}^{+\infty}  \Big( \underline{\bx}_{q}^{2k}(  \Delta_{\bx_p}A_{2k}(\bx_p) -(2k+q)(2k+2) A_{2k+2}(\bx_p)  )\\
  &&\quad +
  \underline{\bx}_{q}^{2k+1} (  \Delta_{\bx_p}A_{2k+1}(\bx_p) -(2k+2) (2k+q+2)A_{2k+3}(\bx_p) )
 \Big).
\end{eqnarray*}
Hence, the harmonicity of $f$ gives that for all $k\in \mathbb{N}$,
\begin{eqnarray*}\label{harmonic-condition}
 \left\{
\begin{array}{ll}
\Delta_{\bx_p}A_{2k}(\bx_p)=(2k+q)(2k+2) A_{2k+2}(\bx_p),
\\
\Delta_{\bx_p}A_{2k+1}(\bx_p)=(2k+2) (2k+q+2)A_{2k+3}(\bx_p),
\end{array}
\right.
\end{eqnarray*}
which imply that
\begin{eqnarray*}
A_{2k}(\bx_p)& =&\frac{1}{2k(2k+q-2)} \Delta_{\bx_p}A_{2k-2}(\bx_p)\\
 &=&\cdots
 \\
&=&\frac{ \Delta_{\bx_p}^{k} A_{0}(\bx_{p})}{(2k)!!(2k+q-2)(2k+q-4)\cdots q}
\\
&=&\frac{\Gamma(\frac{q}{2})}{2^{2k}k!\Gamma(k+\frac{q}{2})}   \Delta_{\bx_p}^{k} A_{0}(\bx_{p}),\end{eqnarray*}
and, similarly,
$$A_{2k+1}(\bx_p)=\frac{\Gamma(\frac{q}{2}+1)}{2^{2k}k!\Gamma(k+\frac{q}{2}+1)}   \Delta_{\bx_p}^{k} A_{1}(\bx_{p}).$$
Consequently,   $f$ has the expansion
\begin{equation}\label{Cauchy-Kovalevska-extension-sum-1-1-harmonic}
 f(\bx)=\sum_{k=0}^{+\infty} \frac{\Gamma(\frac{q}{2})(-1)^{k}r^{2k}}{2^{2k}k!\Gamma(k+\frac{q}{2})}   \Delta_{\bx_p}^{k} A_{0}(\bx_{p})+
\underline{\bx}_q     \sum_{k=0}^{+\infty} \frac{\Gamma(\frac{q}{2}+1)(-1)^{k}r^{2k}}{2^{2k}k!\Gamma(k+\frac{q}{2}+1)} \Delta_{\bx_p}^{k} A_{1}(\bx_{p}),
\end{equation}
thus $f$ has the  closed form in (\ref{Harm-Cauchy-Kovalevska-extension-expression}):
$$f(\bx)=\Gamma(\frac{q}{2})\Big( \frac{r\sqrt{\Delta_{\bx_{p}}}}{2}\Big)^{-\frac{q}{2}} \Big[ \frac{r\sqrt{\Delta_{\bx_{p}}}}{2} J_{\frac{q}{2}-1}(r\sqrt{\Delta_{\bx_{p}}})[A_{0}(\bx_{p})] + \frac{q\underline{\bx}_{q} }{2} J_{\frac{q}{2}}(r\sqrt{\Delta_{\bx_{p}}}) [A_{1}(\bx_{p})] \Big].$$

Finally, we show the  function given by (\ref{Harm-Cauchy-Kovalevska-extension-sum}) satisfies  the initial conditions
 (\ref{initial-conditions-f}).    By  Lemma  \ref{CK-formular-Four}, it holds that, for any function $g(\bx_p)\in C^{1}(\mathbb{R}^{p+1}, \mathbb{A})$,
 $$D_{\underline{\bx}_q} (\underline{\bx}_q^{k}g(\bx_p)) \mid_{\underline{\bx}_{q}}=0, \quad k\geq2,$$
and
$$D_{\underline{\bx}_q} (\underline{\bx}_q g(\bx_p)) =-qg(\bx_p).$$
Hence,
$$ D_{\underline{\bx}_q}f(\bx)\mid_{\underline{\bx}_{q}=0}=-q A_{1}(\bx_{p}).$$
Furthermore, the formula $f(\bx)\mid_{\underline{\bx}_{q}=0}= A_{0}(\bx_{p})$ holds trivially.
The proof is complete.
\end{proof}
\begin{definition}
The function given by the series in  (\ref{Harm-Cauchy-Kovalevska-extension-sum}) is called  harmonic generalized CK-extension of the pair $(A_0,A_1)$ and is  denoted by $HGCK[A_{0},A_1]$.
\end{definition}
We introduce a useful notation.
\begin{definition}\label{even-odd-parts}
Let $f\in   \mathcal{GS} (\Omega_{D}, \mathbb{A}) $.
The function $\mathcal{PE}[f]: \Omega_{D} \longrightarrow \mathbb{A}$, called partial even    part  of $f$, and the   function $\mathcal{PO}[f]: \Omega_{D} \longrightarrow \mathbb{A},$  called partial odd   part  of $f$, are defined as, respectively,
$$\mathcal{PE}[f](\bx)=\frac{1}{2}\big( f(\bx)+f(\bx_{\diamond}) \big), $$
$$\mathcal{PO}[f](\bx)=\frac{1}{2} \big( f(\bx)-f(\bx_{\diamond}) \big).$$
where $\bx=\bx_{p}+\underline{\bx}_{q}$ and $\bx_{\diamond}=\bx_{p}-\underline{\bx}_{q}$.
\end{definition}

We denote the class  of harmonic  generalized partial-slice functions on  the p-symmetric domain $\Omega$ by
\begin{eqnarray}\label{AH}
\mathcal{AH}(\Omega,\mathbb{A}):=\mathrm{Harm}(\Omega,\mathbb{A}) \cap \mathcal{GS}(\Omega,\mathbb{A}).
\end{eqnarray}
\begin{remark}\label{HGCK-two-parts}
Since $f$ in \eqref{Harm-Cauchy-Kovalevska-extension-sum} is a generalized partial-slice function, see Remark \ref{axial}.
Theorem \ref{Harm-generalized-CK-extension} gives a one-to-one correspondence between $( C^{\omega}(\Omega_0, \mathbb{A}) )^{2}$  and   $\mathcal{AH}(\Omega,\mathbb{A})$ as
 $$HGCK: (C^{\omega}(\Omega_0, \mathbb{A}) )^{2} \rightarrow  \mathcal{AH}(\Omega,\mathbb{A}),$$
 $$(A_0,A_1)\mapsto HGCK[A_0,A_1],$$
 where   $\Omega$ is a  p-symmetric slice domain in $M$ and  $\Omega_0=\Omega \cap\mathbb{R}^{p+1}$.\\
Furthermore, the harmonic generalized CK-extension can be split into two parts:
$$HGCK[A_{0},A_1]= HGCK[A_{0},0] +HGCK[0,A_1],$$
where
$$ HGCK[A_{0},0]=\mathcal{PE}\circ HGCK[A_{0},A_1],\
 HGCK[0,A_{1}]=\mathcal{PO}\circ HGCK[A_{0},A_1].$$
\end{remark}
By Theorems \ref{generalized-CK-extension} and  \ref{Harm-generalized-CK-extension}, we infer that

\begin{corollary}  \label{GCK-HGCK}
Let  $\Omega_{0}\subseteq \mathbb{R}^{p+1}$ be a domain and  consider the real analytic  functions $A_{k} \in C^{\omega}(\Omega_0, \mathbb{A})$, $k=0,1$. Then  we have
 $$GCK[A_{0}]= HGCK[A_{0},0] +\frac{1}{q}HGCK[0,D_{\bx_p}A_{0}],$$
and
$$HGCK[A_{0},A_1]= \mathcal{PE}\circ GCK[A_{0}]  + q \mathcal{PO} \circ GCK[f_0].$$
where $f_0$ is a solution of the equation $D_{\bx_p} f_0=A_1$, if it exists.
\end{corollary}

 \begin{proposition}  \label{D-GCK-HGCK}
 	Let  $\Omega_{0}\subseteq \mathbb{R}^{p+1}$ be a domain and  consider the real analytic  function  $f_{0} \in C^{\omega}(\Omega_0, \mathbb{A})$. Then  we have
 	$$\overline{D}_{\bx} \circ HGCK[f_0,0]= GCK \circ \overline{D}_{\bx_p}[f_0], $$
 and
 	$$\overline{D}_{\bx}\circ HGCK[0,f_0]=qGCK[f_0]. $$
 \end{proposition}
Note that  the  operators $\overline{D}_{\bx_p} $ and $GCK$ in general do not commute, unless $p=0$.
\begin{proof}
Let $\Omega_{0}\subseteq \mathbb{R}^{p+1}$ be a domain and $f_{0} \in C^{\omega}(\Omega_0, \mathbb{A})$.
Recalling (\ref{Cauchy-Kovalevska-extension-sum-1-1-harmonic}), (\ref{CK-formular-1-D-bar}) in Lemma  \ref{CK-formular-Four-D}, Lemma \ref{d-dbar-laplace}, and (\ref{Cauchy-Kovalevska-extension-sum-1-1}), we have the following chain of equalities
\begin{eqnarray*}
	&&\overline{D}_{\bx} \circ HGCK[f_0,0] (\bx)\\
	& =&  \overline{D}_{\bx} \Big( \sum_{k=0}^{+\infty} \frac{\Gamma(\frac{q}{2}) \underline{\bx}_q ^{2k}}{2^{2k}k!\Gamma(k+\frac{q}{2})}   \Delta_{\bx_p}^{k}   f_{0}(\bx_{p}) \Big)
\\
&=&    \sum_{k=0}^{+\infty} \frac{\Gamma(\frac{q}{2})}{2^{2k}k!\Gamma(k+\frac{q}{2})}   \big(\underline{\bx}_q ^{2k}(\Delta_{\bx_p}^{k}  \overline{D}_{\bx_p}  f_{0}(\bx_{p}))+2k\underline{\bx}_q ^{2k-1} \Delta_{\bx_p}^{k}  f_{0}(\bx_{p}) \big)
\\
&=&\sum_{k=0}^{+\infty} \frac{\Gamma(\frac{q}{2})\underline{\bx}_q ^{2k}}{2^{2k}k!\Gamma(k+\frac{q}{2})}   \Delta_{\bx_p}^{k}  (\overline{D}_{\bx_p}  f_{0}(\bx_{p}))+ \sum_{k=1}^{+\infty} \frac{\Gamma(\frac{q}{2})\underline{\bx}_q ^{2k-1}}{2^{2k-1}(k-1)!\Gamma(k+\frac{q}{2})}  \Delta_{\bx_p}^{k}  f_{0}(\bx_{p}) \\
&=&\sum_{k=0}^{+\infty} \frac{\Gamma(\frac{q}{2})\underline{\bx}_q ^{2k}}{2^{2k}k!\Gamma(k+\frac{q}{2})}   \Delta_{\bx_p}^{k}  (\overline{D}_{\bx_p}  f_{0}(\bx_{p}))+ \sum_{k=0}^{+\infty} \frac{\Gamma(\frac{q}{2})\underline{\bx}_q ^{2k+1}}{2^{2k+1}k!\Gamma(k+1+\frac{q}{2})}  \Delta_{\bx_p}^{k}
(D_{\bx_p}(\overline{D}_{\bx_p} f_{0}(\bx_{p}))) \\
&=& GCK [\overline{D}_{\bx_p}f_0](\bx). \end{eqnarray*}
Furthermore, from (\ref{Cauchy-Kovalevska-extension-sum-1-1-harmonic}),   (\ref{CK-formular-2-D-bar}) in Lemma  \ref{CK-formular-Four-D}, and (\ref{Cauchy-Kovalevska-extension-sum-1-1}), we have
\begin{eqnarray*}
	&&\overline{D}_{\bx} \circ HGCK[0,f_0] (\bx)\\
	& =&  \overline{D}_{\bx} \Big( \sum_{k=0}^{+\infty} \frac{\Gamma(\frac{q}{2}+1) \underline{\bx}_q ^{2k+1}}{2^{2k}k!\Gamma(k+1+\frac{q}{2})}   \Delta_{\bx_p}^{k}   f_{0}(\bx_{p}) \Big)
\\
&=&    \sum_{k=0}^{+\infty} \frac{\Gamma(\frac{q}{2}+1)}{2^{2k}k!\Gamma(k+1+\frac{q}{2})}   \big(\underline{\bx}_q ^{2k+1}\Delta_{\bx_p}^{k}  (D_{\bx_p}  f_{0}(\bx_{p}))+(2k+q)\underline{\bx}_q ^{2k} \Delta_{\bx_p}^{k}  f_{0}(\bx_{p}) \big)
\\
&=&q   \Big( \sum_{k=0}^{+\infty} \frac{\Gamma(\frac{q}{2})\underline{\bx}_q ^{2k+1}}{2^{2k+1}k!\Gamma(k+1+\frac{q}{2})}   \Delta_{\bx_p}^{k}  (D_{\bx_p}  f_{0}(\bx_{p}))+ \sum_{k=0}^{+\infty} \frac{\Gamma(\frac{q}{2})\underline{\bx}_q ^{2k}}{2^{2k}k!\Gamma(k+\frac{q}{2})}  \Delta_{\bx_p}^{k}  f_{0}(\bx_{p}) \Big)\\
&=& q GCK [ f_0]. \end{eqnarray*}
The proof is complete.
\end{proof}

\section{Fueter-Sce theorem}

In the context of generalized partial-slice monogenic functions, a Fueter-Sce theorem was proved in \cite{Xu-Sabadini-2}.
In \cite{Perotti-22}, Perotti  gave a  version of the Fueter-Sce-Qian theorem for slice regular
functions  over  real alternative $\ast$-algebras. In this section we treat the combined case in which we work with generalized partial-slice regular functions in real alternative $*$-algebras.  Note that, when $\mathbb{A}$ is not associative, all results obtained in \cite{Perotti-22} requires that the  involved functions are  $M$-admissible in the sense of
\cite[Definition 6]{Perotti-22}, while we do not make use of this additional assumption.

In this section, we shall use some symbols and terminologies introduced in the next definition. We use some notations in Section \ref{subs2.5}.
\begin{definition}\label{spherical-value}
Let $f\in   \mathcal{GS} (\Omega_{D}, \mathbb{A}) $.
The function $f^{\circ}_{s}: \Omega_{D} \longrightarrow \mathbb{A}$, called spherical value of $f$, and the   function $f^{\prime}_{s}: \Omega_{D}\setminus \mathbb{R}^{p+1} \longrightarrow \mathbb{A}$,  called spherical derivative of $f$, are defined as, respectively,
$$ f^{\circ}_{s}(\bx)=\frac{1}{2}\big( f(\bx)+f(\bx_{\diamond}) \big), $$
$$ f^{\prime}_{s}(\bx)=\frac{1}{2}\underline{\bx}_{q}^{-1}\big( f(\bx)-f(\bx_{\diamond}) \big).$$
\end{definition}

Let  $f(\bx)=F_1(\bx')+\underline{\omega} F_2(\bx') \in  \mathcal{GS} (\Omega_{D}, \mathbb{A})$. By definition, we have,   for $\bx \in \Omega_{D}$,
\begin{equation}\label{spherical-value-e-o}
\mathcal{PE}[f](\bx)=f^{\circ}_{s}(\bx)= F_1 (\bx')=f^{\circ}_{s}(\bx_{\diamond}),\ \ \ f(\bx)= \mathcal{PE}[f](\bx)+\mathcal{PO}[f](\bx),
\end{equation}
and for $\bx \in \Omega_{D}\setminus \mathbb{R}^{p+1}$
\begin{equation}\label{spherical-value-derivative}f^{\prime}_{s}(\bx)= \frac{1}{r}F_2 (\bx')= f^{\prime}_{s}(\bx_{\diamond}),\ \ \ \mathcal{PO}[f](\bx)=\underline{\bx}_{q} f_{s}^{\prime}(\bx), \ \ \
 f(\bx)=f^{\circ}_{s}(\bx)+ \underline{\bx}_{q} f_{s}^{\prime}(\bx).
 \end{equation}

The spherical Dirac operator $\Gamma$ associated to the basis $\mathcal{B}$ of $\mathbb{R}^{q}$  is defined  by
$$\Gamma f=-\frac{1}{2}\sum_{i, j=p+1}^{p+q} v_{i}(v_{j}(L_{ij}f)),\quad f\in C^{1}(\Omega,\mathbb{A}),$$
where
$ L_{ij}=x_{i}  \partial_{x_{j}}-x_{j}  \partial_{x_{i}}$  are tangential differential operators.
From the proof of Lemma \ref{a-w-b}, we have that for all $a\in \mathbb{A}$
$$v_{i}(v_{j}a)=-v_{j}(v_{i}a),  \quad 1\leq i\neq j\leq m. $$
Hence, we obtain
$$\Gamma f= -\sum_{p+1\leq i < j\leq p+q} v_{i}(v_{j} L_{ij}f).$$

\begin{lemma}  \label{Fueter-Sce-Qian-lemma-1}
Let    $f(\bx)=F_1(\bx')+\underline{\omega} F_2(\bx') \in {\mathcal{GS}}^{2}(\Omega_{D}, \mathbb{A})\cap C^{2}(\Omega_{D}, \mathbb{A})$. Then it holds that for all $\bx\in\Omega_{D}$
\begin{eqnarray}\label{relation-harm}
\Delta_{\bx} f(\bx)= \Delta_{\bx'}F_1(\bx')+ \underline{\omega} \Delta_{\bx'}F_2(\bx')+ (q-1) ( (\frac{1}{r}\partial_{r}) F_1(\bx')+\underline{\omega} (\partial_{r}\frac{1}{r}) F_2(\bx') ),
\end{eqnarray}
where  $\Delta_{\bx'}=\Delta_{\bx_p}+\partial_{r}^{2}$ is the Laplacian  in $\mathbb {R} ^{p+2}$ and the values of $(\frac{1}{r}\partial_{r}) F_1$ and $(\partial_{r}\frac{1}{r}) F_2$ at  $\bx'=(\bx_{p},0)\in  D$   are  meant as
\begin{eqnarray}\label{relation-harm-000}
\lim_{r\rightarrow0} (\frac{1}{r}\partial_{r}) F_1(\bx')=\partial_{r}^{2}F_1(\bx_{p},0),
\end{eqnarray}
and
\begin{eqnarray}\label{relation-harm-0001}
\lim_{r\rightarrow0}(\partial_{r}\frac{1}{r}) F_2(\bx')=\frac{1}{2}\partial_{r}^{2}F_2(\bx_{p},0),
\end{eqnarray}
respectively.
\end{lemma}
\begin{proof}
Let    $f(\bx)=F_1(\bx')+\underline{\omega} F_2(\bx') \in {\mathcal{GS}}^{2}(\Omega_{D}, \mathbb{A})\cap C^{2}(\Omega_{D}, \mathbb{A})$.
First, it should be  pointed  out that   the existence of limits in  (\ref{relation-harm-000}) and (\ref{relation-harm-0001}) depends on   the  even-oddness of $(F_1,F_2)$    in the $(p+2)$-th variable, which   can be justified  with the standard arguments. Then we prove (\ref{relation-harm}) for  $\bx\in \Omega_{D}\setminus \mathbb{R}^{p+1}$ so that the case  $\bx=(\bx_{p},0)\in\Omega_{D}$  follows immediately from  the  continuity.

In fact, (\ref{relation-harm}) can be obtained   by showing that
\begin{eqnarray}\label{relation-harm-1}
\Delta_{\bx} f^{\circ}_{s}(\bx)= \Delta_{\bx'}F_1(\bx')+  (q-1) (\frac{1}{r}\partial_{r}) F_1(\bx'),
\end{eqnarray}
and
\begin{eqnarray}\label{relation-harm-2}
\Delta_{\bx} (\underline{\bx}_{q}f^{\prime}_{s}(\bx))=  \underline{\omega} ( \Delta_{\bx'}F_2(\bx')+ (q-1) (\partial_{r}\frac{1}{r}) F_2(\bx') ).
\end{eqnarray}
To this end, consider the decomposition
$$\Delta_{\bx}=\Delta_{\bx_p}+\Delta_{\underline{\bx}_q}, \quad \Delta_{\bx_p}=\sum_{i=0}^{p} \partial_{x_i}^{2}, \quad \Delta_{\underline{\bx}_q}=\sum_{i=p+1}^{p+q} \partial_{x_i}^{2},$$
where $\Delta_{\underline{\bx}_q}$  can be represented in  spherical coordinates as
$$\Delta_{\underline{\bx}_q}=\partial_{r}^{2}+\frac{q-1}{r}\partial_{r}+\frac{1}{r^{2}} \Delta_{\underline{\omega}},$$
and $\Delta_{\underline{\omega}}$ is    the Laplace-Beltrami operator on  $\mathbb{S}$, see e.g. \cite[Lemma 1.4.1]{Dai-Xu} for more details.
Note that     $\Delta_{\underline{\omega}}  f^{\circ}_{s}(\bx)=0$, and so (\ref{relation-harm-1}) holds immediately. To prove (\ref{relation-harm-2}), we
observe  that $$\Delta_{\bx} \underline{\bx}_{q}=0,\ \Delta_{\underline{\omega}} f^{\prime}_{s}(\bx)=0,  $$  we have
 \begin{eqnarray*}
\Delta_{\bx} (\underline{\bx}_{q}f^{\prime}_{s}(\bx))
&=&2\sum_{i=0}^{p+q} ( \partial_{x_i} \underline{\bx}_{q} ) ( \partial_{x_i}f^{\prime}_{s}(\bx))+ \underline{\bx}_{q} (\Delta_{\bx} f^{\prime}_{s}(\bx)) \\
&=&2\sum_{i=p+1}^{p+q} v_i  \Big( \frac{x_i}{r} \partial_r \frac{ F_2(\bx')}{r}\Big)+ \underline{\bx}_{q} (\Delta_{\bx} f^{\prime}_{s}(\bx)) \\
&=& 2\sum_{i=p+1}^{p+q} (v_ix_i)   \frac{1} {r} \partial_r \frac{ F_2(\bx')}{r} + \underline{\bx}_{q} (\Delta_{\bx} f^{\prime}_{s}(\bx))
 \\
&=&2 \frac{\underline{\bx}_q}{r}   \partial_r \frac{ F_2(\bx')}{r}+ \underline{\bx}_{q} \Big((\Delta_{\bx_p} +\partial_{r}^{2} +\frac{q-1}{r}\partial_{r}) \frac{ F_2(\bx')}{r}\Big)\\
&=&\underline{\omega} ( \Delta_{\bx'}F_2(\bx')+ (q-1) (\partial_{r}\frac{1}{r}) F_2(\bx') ),
\end{eqnarray*}
where the last equality follows from the formula
$$ \partial_{r}^{2} \frac{ F_2(\bx')}{r}=-\frac{2}{r} \partial_{r}  \frac{ F_2(\bx')}{r}+\frac{1}{r} \partial_{r}^{2}F_2(\bx'),$$
and the proof is complete.
\end{proof}

\begin{proposition}  \label{GSR-Harmonic}
  Let $f(\bx)=F_1(\bx')+\underline{\omega} F_2(\bx') \in  \mathcal{GSR}  (\Omega_{D},\mathbb{ A})$. Then
  $F_1, F_2$  are   real analytic on $D, f$ is real analytic on $\Omega_{D}$, and
    \[\Delta_{\bx'}F_1 =0,\quad \Delta_{\bx'}F_2 =0,\]
  where $\Delta_{\bx'}$ is the Laplacian  in $\mathbb {R} ^{p+2}$.
\end{proposition}
\begin{proof}
Let $f(\bx)=F_1(\bx')+\underline{\omega} F_2(\bx') \in {\mathcal{GSR}}(\Omega_{D},\mathbb{ A})$.
From   Theorem  \ref{relation-GSR-GSM}, it follows that  $f \in {\mathcal{GSM}}(\Omega_{D},\mathbb{ A})$. Hence, by Proposition \ref{GSM-Harm},  we know that $F_1, F_2\in C^{2}(D, \mathbb{A})$.   Then, by (\ref{C-R}) and Lemma \ref{d-dbar-laplace}, we have
$$\Delta_{\bx_{p}}  F_1(\bx')
 =\overline{D}_{\bx_p} (D_{\bx_p} F_1(\bx'))=\overline{D}_{\bx_p}   (\partial_{r} F_2 (\bx'))=\partial_{r}\overline{D}_{\bx_p}   F_2(\bx') =-\partial_{r}^{2} F_1(\bx'),$$
 which gives $\Delta_{\bx'}F_1 =0$.  Now  the harmonicity of $F_1$ guarantees that it is also  real analytic  on $D$.
 Similarly, we can prove $\Delta_{\bx'}F_2 =0$ as
$$\Delta_{\bx_{p}}  F_2(\bx')
 =D_{\bx_p} (\overline{D}_{\bx_p}  F_2(\bx'))=-D_{\bx_p}  (  \partial_{r} F_1(\bx')) =-\partial_{r}D_{\bx_p}    F_1 (\bx')=-\partial_{r}^{2} F_2(\bx').$$

Now it remains to prove that $f$   is real analytic on $\Omega_{D}$. To this end, recalling that $(F_1,F_2)$  is an even-odd pair in the $(p+2)$-th variable, from \cite[Theorem 1]{Whitney} and \cite[Theorem 2]{Whitney}, there exist real analytic functions $G_{1}$ and $G_{2}$  such that
$$F_1(\bx_{p}, r)=G_{1}(\bx_{p}, r^{2}), \ F_2(\bx_{p}, r)=rG_{2}(\bx_{p}, r^{2}), \quad (\bx_{p}, r)\in D.$$
Then the function
$$f(\bx)=f(\bx_p+\underline{\bx}_q)=G_{1}(\bx_{p}, r^{2})+ \underline{\bx}_q G_{2}(\bx_{p}, r^{2}), \quad r=|\underline{\bx}_q|,$$
is real analytic on $\Omega_{D}$,  which completes the proof.
\end{proof}

  \begin{remark}\label{0-lim}
 Let $f(\bx)=F_1(\bx')+\underline{\omega} F_2(\bx') \in {\mathcal{GSR}}(\Omega_{D})$. Then,  for all $k \in \mathbb{N}$,  Proposition \ref{GSR-Harmonic}  guarantees the   functions
 \begin{equation}\label{r}
(\frac{1}{r}\partial_{r})^{k} F_1(\bx'), \ (\partial_{r}\frac{1}{r})^{k}F_2(\bx')
\end{equation}
 are well-defined for  $\bx'=(\bx_{p},r) \in  D$ with $r\neq0$.
  Reasoning as in the proof of Lemma \ref{Fueter-Sce-Qian-lemma-1}, one can deduce that the following  two limits exist
   \begin{equation}\label{r-lim}
\lim_{r\rightarrow0} (\frac{1}{r}\partial_{r})^{k} F_1(\bx'), \qquad \lim_{r\rightarrow0}(\partial_{r}\frac{1}{r})^{k}F_2(\bx').
\end{equation}
{\bf Notation.} For the sake of simplicity, in Lemma  \ref{lemma-inter-relation}, Lemma \ref{Fueter-Sce-Qian-lemma-k}, and  Theorem \ref{Fueter-theorem} below,  we will use the convention that the values of  functions in (\ref{r}) at $(\bx_{p},0)\in  D$  are  meant as limits in (\ref{r-lim}), respectively.
  \end{remark}

\begin{lemma} \label{lemma-inter-relation}
Let $f(\bx)=F_1(\bx')+\underline{\omega} F_2(\bx') \in  \mathcal{GSR} (\Omega_{D}, \mathbb{A})$ and denote for $k \in \mathbb{N}$
\begin{equation}\label{A-k-B-k}
 A_{k}(\bx')=   (\frac{1}{r}\partial_{r})^{k}F_1(\bx'), \quad
 B_{k}(\bx')=    (\partial_{r}\frac{1}{r})^{k}F_2(\bx').
\end{equation}
  Then it holds that, for all $k \in \mathbb{N}$ and $\bx\in \Omega_{D}$,
  \begin{eqnarray}\label{inter-relation}
 \left\{
\begin{array}{ll}
D_{\bx_p}  A_{k}(\bx')- \partial_{r} B_{k}(\bx')=\frac{2k}{r}B_{k}(\bx'),
\\
 \overline{D}_{\bx_p} B_{k}(\bx') + \partial_{r} A_{k}(\bx')=0.
\end{array}
\right.
\end{eqnarray}
 \end{lemma}
\begin{proof}For $k=0$, the system (\ref{inter-relation}) coincides with the generalized Cauchy-Riemann equations in (\ref{C-R}); see Definition \ref{definition-GSR}. The step from $k-1$ to $k$ for (\ref{inter-relation}) can be obtained as follows:
 \begin{eqnarray*}
 D_{\bx_p}  A_{k}(\bx')
&=&\frac{1}{r} \partial_{r} D_{\bx_p} A_{k-1}(\bx')\\
&=& \frac{1}{r} \partial_{r}  (\partial_{r} B_{k-1}(\bx')+\frac{2k-2}{r}B_{k-1}(\bx'))
 \\
 &=& \frac{1}{r}   \partial_{r}^{2} B_{k-1}(\bx')+\frac{2k-2}{r}B_{k}(\bx')
 \\
&=&\partial_{r} B_{k}(\bx')+\frac{2k}{r}B_{k}(\bx'),
\end{eqnarray*}
and
 \begin{eqnarray*}
 \overline{D}_{\bx_p} B_{k}(\bx')
&=&\overline{D}_{\bx_p} ( \partial_{r} \frac{B_{k-1}(\bx')}{r}) \\
&=& \partial_{r} \frac{\overline{D}_{\bx_p} B_{k-1}(\bx')}{r}
 \\
&=&\partial_{r} \frac{- \partial_{r} A_{k-1}(\bx') }{r}\\
&=&- \partial_{r} A_{k}(\bx'),
\end{eqnarray*}
as asserted.

 \end{proof}

Now we need to establish  the following results  for generalized partial-slice functions, which generalize \cite[Proposition 5.3]{Xu-Sabadini}
from Clifford algebras to  general real alternative $\ast$-algebras.
We refer to the works \cite[Proposition 3.3.2]{Perotti-19} in the framework of  Clifford algebras  and   \cite[Theorem 8 and Proposition 9]{Perotti-22} for $M$-admissible functions over   real alternative $\ast$-algebras.  It should be pointed out, one more time, that for our results we do not impose the additional $M$-admissible condition.
\begin{proposition}\label{relation-Dirac-and-Dirac-slice}
Let $f \in {\mathcal{GS}}^{1}(\Omega_{D}, \mathbb{A}) \cap C^{1}(\Omega_{D}, \mathbb{A}) $. Then the following two formulas hold on $\Omega_{D}\setminus \mathbb{R}^{p+1}$:\\
(i) $\Gamma f(\bx)=(q-1)   \underline{\bx}_{q}f^{\prime}_{s}(\bx)$;\\
(ii) $(D_{\bx}-{D_{\underline{\omega}})}f(\bx) =(1-q)f^{\prime}_{s}(\bx)$.
\end{proposition}
\begin{proof}
First of all, by the Artin's theorem  and Moufang identities,  we have,  for all $a\in \mathbb{A}$ and $i,j=1,\ldots,p+q$ with $i\neq j$,
\begin{equation}\label{vvv-a-1}
 v_{i}(v_{j}( v_ja))= v_{i}((v_{j}  v_j)a)= - v_{i}a,
\end{equation}
and
\begin{equation}\label{vvv-a-2}
v_{i}(v_{j} (v_i a))= (v_{i}v_{j}  v_i )a=-(v_{j} v_{i} v_i )a=v_{j}  a.\end{equation}
Let $f(\bx)=F_1(\bx')+\underline{\omega} F_2(\bx')\in {\mathcal{GS}}^{1}(\Omega_{D}, \mathbb{A}) \cap C^{1}(\Omega_{D}, \mathbb{A})$.
Then it holds that for $x\in \Omega_{D}\setminus \mathbb{R}^{p+1}$
 $$L_{ij} f^{\circ}_{s}(\bx)=L_{ij} f^{\prime}_{s}(\bx)=0.$$
Therefore, by (\ref{vvv-a-1}) and (\ref{vvv-a-2}),
 \begin{eqnarray*}
\Gamma f(\bx)
 &=& \Gamma (\underline{\bx}_{q}f^{\prime}_{s}(\bx))
 \\
 &=& -\frac{1}{2}\sum_{p+1\leq i\neq j\leq p+q} v_{i}(v_{j}((x_iv_j-x_jv_i) f^{\prime}_{s}(\bx)))
 \\
  &=&\frac{1}{2}\Big(\sum_{p+1\leq i\neq j\leq p+q} (x_i v_{i} +x_jv_j)\Big)f^{\prime}_{s}(\bx)
 \\
 &=&(q-1) \underline{\bx}_{q}f^{\prime}_{s}(\bx),
\end{eqnarray*}
which proves (i).

To show (ii), we first  show  the following formula
\begin{equation}\label{Dq-F}
	D_{\underline{\bx}_q} (\underline{\omega} F_2(\bx')) =(1-q)  f^{\prime}_{s}(\bx)-  \partial_{r} F_2(\bx').\end{equation}
In fact, it holds that by direct  calculations and  by  the Artin's theorem
 \begin{eqnarray*}
 D_{\underline{\bx}_q} (\underline{\omega} F_2(\bx'))
 &=& D_{\underline{\bx}_q} (\underline{\bx}_q f^{\prime}_{s}(\bx))
 \\
 &=& \sum_{i=p+1}^{p+q}  v_i \Big( (\partial_{x_i}\underline{\bx}_q) f^{\prime}_{s}(\bx)+\underline{\bx}_q ( \partial_{x_i}f^{\prime}_{s}(\bx)) \Big)
 \\
 &=& \sum_{i=p+1}^{p+q}  v_i \Big( v_i f^{\prime}_{s}(\bx)+\underline{\bx}_q \big( \frac{x_i}{r}\partial_{r}f^{\prime}_{s}(\bx) \big) \Big)
 \\
 &=&  \sum_{i=p+1}^{p+q} v_i^{2}  f^{\prime}_{s}(\bx)+  \sum_{i=p+1}^{p+q} (v_i \frac{x_i}{r}) (\underline{\bx}_q \partial_{r}f^{\prime}_{s}(\bx) )
 \\
 &=&  -q  f^{\prime}_{s}(\bx)+\frac{ \underline{\bx}_q }{r}(\underline{\bx}_q \partial_{r}f^{\prime}_{s}(\bx) )
 \\
 &=&  -q  f^{\prime}_{s}(\bx)-r \partial_{r}f^{\prime}_{s}(\bx)
  \\
 &=& (1-q)  f^{\prime}_{s}(\bx)-  \partial_{r} F_2(\bx').
 \end{eqnarray*}
Hence, using again the Artin's theorem and (\ref{Dq-F}), we get
 \begin{eqnarray*}
 (D_{\bx}- D_{\underline{\omega}} ) f(\bx)
 &=& (D_{\underline{\bx}_q}-\underline{\omega}\partial_{r})(F_1(\bx')+\underline{\omega} F_2(\bx'))
 \\
 &=& D_{\underline{\bx}_q} F_1(\bx') -\underline{\omega} \partial_{r} F_1(\bx') +
 D_{\underline{\bx}_q} (\underline{\omega} F_2(\bx')) -    \underline{\omega} (\underline{\omega}\partial_{r} F_2(\bx'))
 \\
 &=& \frac{\underline{\bx}_q}{r}   \partial_r F_1(\bx')  -\underline{\omega} \partial_{r} F_1(\bx')+
 D_{\underline{\bx}_q} (\underline{\omega} F_2(\bx')) + \partial_{r} F_2(\bx')
 \\
 &=&(1-q)  f^{\prime}_{s}(\bx),
\end{eqnarray*}
which completes the proof.
\end{proof}

\begin{lemma}  \label{Fueter-Sce-Qian-lemma-k}
Let $f(\bx)=F_1(\bx')+\underline{\omega} F_2(\bx') \in  \mathcal{GSR} (\Omega_{D}, \mathbb{A})$. Then it holds that, for all $k \in \mathbb{N}$ and $\bx\in \Omega_{D}$,
\begin{equation}\label{k}
 \Delta_{\bx}^{k}f(\bx)= C_{q}(k)( A_{k}(\bx') +\underline{\omega} B_{k}(\bx') ),
\end{equation}
and
\begin{equation}\label{d-k}
D_{\bx} \Delta_{\bx}^{k}f(\bx)= -C_{q}(k+1) \frac{ B_{k}(\bx')}{r},
\end{equation}
where  $C_{q}(k):=(q-1)(q-3)\cdots(q-2k+1)$ for $k\in \mathbb{N}\setminus\{0\}$ and $C_{q}(0)=1$, and $A_{k}, B_{k}$ are given by (\ref{A-k-B-k}).
\end{lemma}
\begin{proof}
Let $f(\bx)=F_1(\bx')+\underline{\omega} F_2(\bx') \in  \mathcal{GSR} (\Omega_{D}, \mathbb{A})$.
 In view of Proposition \ref{GSR-Harmonic}  and Remark \ref{0-lim}, both sides of  (\ref{k}) and (\ref{d-k}) are  well-defined.   As in the proof of Lemma \ref{Fueter-Sce-Qian-lemma-1},  we only prove assertions  for  $\bx\in \Omega_{D}\setminus \mathbb{R}^{p+1}$ since  the case of $\bx=(\bx_{p},0)\in\Omega_{D}$  follows by the continuity. We prove (\ref{k}) by induction.  By Lemma    \ref{Fueter-Sce-Qian-lemma-1} and Proposition \ref{GSR-Harmonic}, we have
 $$\Delta_{\bx} f(\bx)=   (q-1) ( (\frac{1}{r}\partial_{r}) F_1(\bx')+\underline{\omega} (\partial_{r}\frac{1}{r}) F_2(\bx') ),$$
which shows (\ref{k}) for $k=1$.

Now we suppose that (\ref{k}) holds for $k-1$, i.e.,
\begin{equation}\label{k-1}
 \Delta_{\bx}^{k-1}f(\bx)=C_{q}(k-1)(  ( A_{k-1}(\bx') +\underline{\omega} B_{k-1}(\bx') ) ),
\end{equation}
where $A_{k-1}, B_{k-1}$ are given by (\ref{A-k-B-k}) and we show that (\ref{k}) holds for $k$.  By Lemma \ref{d-dbar-laplace} and  (\ref{inter-relation}) in Lemma  \ref{lemma-inter-relation}, we have
 \begin{eqnarray*}
 \Delta_{\bx_{p}}  A_{k-1}(\bx')
&=&\overline{D}_{\bx_p} (D_{\bx_p}  A_{k-1}(\bx'))\\
&=& \overline{D}_{\bx_p} (  \partial_{r}  +\frac{2k-2}{r})B_{k-1}(\bx')
 \\
&=&(  \partial_{r}  +\frac{2k-2}{r})\overline{D}_{\bx_p}B_{k-1}(\bx')\\
&=&(  \partial_{r}  +\frac{2k-2}{r}) (-\partial_{r}  A_{k-1}(\bx'))
\\
&=&-\partial_{r}^{2} A_{k-1}(\bx')-(2k-2)A_{k}(\bx'),
\end{eqnarray*}
and
 \begin{eqnarray*}
 \Delta_{\bx_{p}}  B_{k-1}(\bx')
&=&D_{\bx_p} (\overline{D}_{\bx_p} B_{k-1}(\bx'))\\
&=&D_{\bx_p} (- \partial_{r} A_{k-1}(\bx') )
 \\
&=&-\partial_{r}^{2} B_{k-1}(\bx')-(2k-2)\partial_{r} \frac{B_{k-1}(\bx')}{r}
\\
&=& -\partial_{r}^{2} B_{k-1}(\bx')-(2k-2) B_{k}(\bx').
\end{eqnarray*}
Hence,  combining with Lemma  \ref{Fueter-Sce-Qian-lemma-1}, we have
\begin{eqnarray*}
\frac{\Delta_{\bx}^{k}f(\bx)}{C_{q}(k-1)}
&=&\Delta_{\bx'} A_{k-1}(\bx')+\underline{\omega} \Delta_{\bx'} B_{k-1}(\bx') + (q-1)(\frac{1}{r}\partial_{r} A_{k-1}(\bx') +\underline{\omega} \partial_{r} \frac{B_{k-1}(\bx') }{r})
\\
&=&-(2k-2)(A_{k}(\bx') + \underline{\omega}  B_{k}(\bx')) +(q-1) (  A_{k}(\bx') +\underline{\omega}  B_{k}(\bx') )
\\
&=&(q-2k+1) (  A_{k}(\bx') +\underline{\omega}  B_{k}(\bx') ),
\end{eqnarray*}
which gives (\ref{k}).

 Now we can use  (\ref{k}), Lemma  \ref{lemma-inter-relation}, and  Proposition \ref{relation-Dirac-and-Dirac-slice}  to obtain (\ref{d-k}) as follows:
\begin{eqnarray*}
&&D_{\bx} \Delta_{\bx}^{k}f(\bx)\\
&=&C_{q}(k) \Big(D_{\underline{\omega}} ( A_{k}(\bx') +\underline{\omega} B_{k}(\bx') ) + (1-q)\frac {B_{k}(\bx')}{r} \Big) \\
&=&C_{q}(k)\Big(D_{\bx_p} A_{k}(\bx')- \partial_r B_{k}(\bx')+ \underline{\omega} (\overline{D}_{\bx_p} B_{k}(\bx') + \partial_{r} A_{k}(\bx') )+(1-q)\frac{B_{k}(\bx')}{r}\Big)
 \\
&=&C_{q}(k) (2k+1-q) \frac{B_{k}(\bx')}{r}\\
&=& -C_{q}(k+1) \frac{ B_{k}(\bx')}{r}.
\end{eqnarray*}
The proof is complete.
\end{proof}

In the next result we prove that, in the setting of generalized partial-slice  functions over
 real alternative $\ast$-algebras,     the  condition  of  monogenicity can be described by  a Vekua-type system as follows.
\begin{lemma}\label{GS-monogenic-lemma}
Let $D \subseteq \mathbb{R}^{p+2}$  be a domain invariant under the reflection with respect the $(p+2)$-th variable, and let $f(\bx)=F_1(\bx')+\underline{\omega} F_2(\bx')  \in {\mathcal{GS}}^{1}(\Omega_{D}, \mathbb{A}) \cap C^{1}(\Omega_{D}, \mathbb{A})$, where $\bx=\bx_p+r\underline{\omega}\in \Omega_{D}$, $\bx'=(\bx_p, r)\in D$. Then $f \in \mathcal {M}(\Omega_{D}, \mathbb{A}) $ if and only if  the components $F_1,F_2$ of
$f$ satisfy the system
 \begin{eqnarray}\label{GS-is-Monogenic}
 \left\{
\begin{array}{ll}
D_{\bx_p}  F_1(\bx')- \partial_{r} F_2(\bx')=\frac{q-1}{r}F_2(\bx'),
\\
 \overline{D}_{\bx_p}  F_2(\bx')+ \partial_{r} F_1(\bx')=0,
\end{array}\quad \quad \bx'\in  D \setminus \mathbb{R}^{p+1}.
\right.
\end{eqnarray}
\end{lemma}

\begin{proof}
Let $f(\bx)=f(\bx_p+r\underline{\omega})=F_1(\bx')+\underline{\omega} F_2(\bx')  \in {\mathcal{GS}}^{1}(\Omega_{D}, \mathbb{A})\cap C^{1}(\Omega_{D}, \mathbb{A})$.   Recalling the Artin's theorem and Lemma \ref{a-w-b}, it holds that
$$D_{\underline{\omega}}f(\bx)  = D_{\bx_p}  F_1(\bx')- \partial_{r} F_2(\bx')+ \underline{\omega}(\overline{D}_{\bx_p} F_2(\bx')  +  \partial_{r} F_1(\bx')).$$
Hence, we deduce that  from  Proposition \ref{relation-Dirac-and-Dirac-slice} (ii)
\begin{eqnarray}\label{relation-slice-and-slice-monogenic}
D_{\bx} f(\bx) =D_{\bx_p}  F_1(\bx')- \partial_{r} F_2(\bx')+\frac{1-q}{r}F_2(\bx')+\underline{\omega}( \overline{D}_{\bx_p}  F_2(\bx')+ \partial_{r} F_1(\bx')),
\end{eqnarray}
 for all $\bx\in\Omega_{D}\setminus \mathbb{R}^{p+1}$.

Suppose that  $f \in \mathcal {M}(\Omega_{D}, \mathbb{A}) $. Then  we deduce that the system \eqref{GS-is-Monogenic} holds  by the arbitrariness of $\underline{\omega}\in\mathbb{S}$ in (\ref{relation-slice-and-slice-monogenic}).

 Conversely, if the components  $F_1,F_2$ of $f(\bx)=F_1(\bx')+\underline{\omega} F_2(\bx')  \in {\mathcal{GS}}^{1}(\Omega_{D}, \mathbb{A})\cap C^{1}(\Omega_{D}, \mathbb{A})$   satisfy   (\ref{GS-is-Monogenic}). Then it follows from (\ref{relation-slice-and-slice-monogenic})  that  $D_{\bx} f=0$ on $\Omega_{D}\setminus \mathbb{R}^{p+1}$. In view of that  $\Omega_{D}\setminus \mathbb{R}^{p+1}$ is dense in $\Omega_D$ and  $f\in C^1(\Omega_D, \mathbb{A})$, we deduce that  $D_{\bx} f=0$ for all $\bx\in \Omega_{D}$, which completes the proof.
\end{proof}

With all the previous results we can now prove the main result  in this section.
\begin{theorem}[Fueter-Sce  theorem]\label{Fueter-theorem}
Let     $f(\bx)=F_1(\bx')+\underline{\omega} F_2(\bx') \in \mathcal{GSR}(\Omega_{D}, \mathbb{A})$.  Then, for odd $q\in \mathbb{N}$,
the function $ \Delta_{\bx}^{\frac{q-1}{2}}f(\bx)$ is monogenic in $\Omega_{D}$ and is given by
\begin{eqnarray}\label{q-1-formula}
 \Delta_{\bx}^{\frac{q-1}{2}}f(\bx)=(q-1)!! (A(\bx')+\underline{\omega} B(\bx')),
\end{eqnarray}
 where
\begin{eqnarray}\label{q-1-formula-A-B}
A(\bx') =  (\frac{1}{r}\partial_{r})^{\frac{q-1}{2}}F_1(\bx'),\quad  B(\bx')=  (\partial_{r}\frac{1}{r})^{\frac{q-1}{2}}F_2(\bx').
 \end{eqnarray}
\end{theorem} \begin{proof}
Let     $f(\bx)=F_1(\bx')+\underline{\omega} F_2(\bx') \in \mathcal{GSR}(\Omega_{D}, \mathbb{A})$.
By (\ref{k}) in  Lemma     \ref{Fueter-Sce-Qian-lemma-k}, the function $\Delta_{\bx}^{\frac{q-1}{2}}f(\bx)$  takes the form  (\ref{q-1-formula}). Then,  by Lemma  \ref{GS-monogenic-lemma},   it remains  to   verify that the pair $(A, B)$ given by (\ref{q-1-formula-A-B}) satisfies   (\ref{GS-is-Monogenic}), which is exactly (\ref{inter-relation}) with $k=\frac{q-1}{2}$ in Lemma  \ref{lemma-inter-relation}. The proof is complete.
 \end{proof}
   When $\mathbb{A}=\mathbb{H}=M,$   Theorem \ref{Fueter-theorem}  was originally obtained by Fueter  \cite{Fueter}.  The Fueter theorem  was generalized by Sce  \cite{Sce} (see an English translation \cite{Colombo-Sabadini-Struppa-20}) in 1957 to the case of Theorem \ref{Fueter-theorem} for  $\mathbb{A}=\mathbb{R}_{0,m}, M=\mathbb{R}^{m+1},$  $(p,q)=(0,m)$  with odd $m$  and   by Qian  \cite{Qian-97} in 1997   for
 even $m$. See e.g. \cite{Qian-15} and references therein for more details of the Fueter-Sce-Qian theorem.

\section{Connections  between the Fueter-Sce theorem and  CK-extensions}
In this section, we shall establish some connections  between the Fueter-Sce theorem and the three types of  CK-extensions given by Section 3.

We denote the monogenic  class   of generalized partial-slice functions on  the p-symmetric domain $\Omega$ by
$$\mathcal{AM}(\Omega,\mathbb{A}):=\mathcal{M}(\Omega,\mathbb{A}) \cap \mathcal{GS}(\Omega,\mathbb{A}),$$
and by $\mathcal{A}(\Omega_{0}, \mathbb{A})$  the set of all $\mathbb{A}$-valued  real analytic  functions defined in $ \Omega_{0} \subseteq \mathbb{R}^{p+1}$ with unique generalized partial-slice monogenic  extension to $\Omega$.
\begin{theorem}\label{CK-Fueter-relation-M}
Let $\Omega\subseteq M$ be a p-symmetric slice domain and $f:\Omega\rightarrow \mathbb{A}$ be a  generalized partial-slice monogenic function.    Then,  for odd $q\in \mathbb{N}$,
we have
 $$\Delta_{\bx}^{\frac{q-1}{2}} f(\bx)
 =\gamma_q GCK [ \Delta_{\bx_{p}}^{\frac{q-1}{2}}f(\bx_{p}) ],$$
which gives the  commutative diagram
$$\xymatrix{
 \mathcal{A}(\Omega_{0}, \mathbb{A}) \ar[d]_{ \gamma_q \Delta_{\bx_{p}}^{\frac{q-1}{2}}} \ar[r]^{CK} & \mathcal{GSM}(\Omega, \mathbb{A}) \ar[d]^{\Delta_{\bx}^{\frac{q-1}{2}}} \\
  \mathcal{A}(\Omega_{0}, \mathbb{A}) \ar[r]^{GCK} & \mathcal{AM}(\Omega, \mathbb{A}),}$$
 where $\gamma_q$ is given by
 \begin{equation}\label{gamma-q}
\gamma_q:=(-1)^{\frac{q-1}{2}}\frac{(q-1)!!}{(q-2)!!}.\end{equation} \end{theorem}
  \begin{proof}
  Let     $f  \in \mathcal{GSM}(\Omega, \mathbb{A})$, where  $\Omega\subseteq M$ is a p-symmetric slice domain.
 From Theorem    \ref{relation-GSR-GSM}, we can set    $f(\bx)=F_1(\bx')+\underline{\omega} F_2(\bx') \in \mathcal{GSR}(\Omega_{D}, \mathbb{A})$.
By Theorem \ref{Fueter-theorem}, we have
  $$\Delta_{\bx}^{\frac{q-1}{2}}f(\bx)=(q-1)!! (A(\bx')+\underline{\omega} B(\bx'))  \in \mathcal {AM}(\Omega, \mathbb{A}),$$
 where $A, B$  are given by (\ref{q-1-formula-A-B}).
Note that $B(\bx_{p}, r)$ is odd with respect to  the variable $r$, so we  obtain that
 $$\Delta_{\bx}^{\frac{q-1}{2}}f(\bx) \mid_{\underline{\bx}_{q}=0}=(q-1)!! A(\bx_p,0) =(q-1)!!  (\frac{1}{r}\partial_{r})^{\frac{q-1}{2}}[F_1](\bx_p,0),$$
 where $ F_1$ is given by, from  Theorems \ref{slice-Cauchy-Kovalevsky-extension} and \ref{Identity-theorem},
 $$F_1(\bx')= \sum_{k=0}^{+\infty} \frac{r^{2k}}{(2k)!}  (-\Delta_{\bx_p})^{k} f(\bx_{p})=
\sum_{k=0}^{+\infty} \frac{r^{2k}}{(2k)!}  (-\Delta_{\bx_p})^{k} [F_1](\bx_{p},0 ).$$
Through direct calculations, we have
 $$(\frac{1}{r}\partial_{r})^{\frac{q-1}{2}}[F_1](\bx_{p},0)= \frac{(-1)^{\frac{q-1}{2}}}{(q-2)!!}  \Delta_{\bx_p}^{\frac{q-1}{2}} f(\bx_{p}).$$
By Theorem \ref{generalized-CK-extension}, the   monogenic function $\Delta_{\bx}^{\frac{q-1}{2}}f(\bx)$ is completely determined by its restriction to $\mathbb{R}^{p+1}$. Hence,
 $$\Delta_{\bx}^{\frac{q-1}{2}} f(\bx)= (-1)^{\frac{q-1}{2}}\frac{(q-1)!!}{(q-2)!!} GCK
   [ \Delta_{\bx_{p}}^{\frac{q-1}{2}}f(\bx_{p}) ](\bx)=\gamma_q GCK [ \Delta_{\bx_{p}}^{\frac{q-1}{2}}f(\bx_{p}) ](\bx),$$
which completes the proof.
\end{proof}
Recalling the definition of $\mathcal{AH}(\Omega, \mathbb{A})$ in (\ref{AH}) we can state and prove the following result.
\begin{theorem}\label{CK-P-E-O}
Let   $\Omega_{0}$ be a domain in  $\mathbb{R}^{p+1}$ and  consider the real analytic function $f_{0}\in C^{\omega}(\Omega_0, \mathbb{A})$.  Then,  for odd $q\in\mathbb{N}$,  the partial  even and odd parts of   $\Delta_{\bx}^{\frac{q-1}{2}}\circ CK[f_0] $ are harmonic  and generalized partial-slice functions in an $(m +1)$-dimensional p-symmetric slice domain $\Omega$ in $M$ with $\Omega_0\subset \Omega$, which give
$$\xymatrix@C=0.5cm{
    & \mathcal{A}(\Omega_{0}, \mathbb{A}) \ar[rr]^{CK} && \mathcal{GSM}(\Omega, \mathbb{A}) \ar[rr]^{\mathcal{PE}\circ  \Delta_{\bx}^{\frac{q-1}{2}}}
     &&\mathcal{AH}(\Omega, \mathbb{A}), }$$
 and
    $$\xymatrix@C=0.5cm{
    & \mathcal{A}(\Omega_{0}, \mathbb{A}) \ar[rr]^{CK} && \mathcal{GSM}(\Omega, \mathbb{A}) \ar[rr]^{\mathcal{PO}\circ  \Delta_{\bx}^{\frac{q-1}{2}}} &&\mathcal{AH}(\Omega, \mathbb{A})}.$$ \end{theorem}
\begin{proof}
Let $f_{0}\in C^{\omega}(\Omega_0, \mathbb{A})$.  By Theorem \ref{slice-Cauchy-Kovalevsky-extension}, $f=CK[f_0]$ is   generalized partial-slice monogenic   in an $(m+1)$-dimensional p-symmetric slice domain $\Omega$ in $M$ with $\Omega_0\subset \Omega$.
From    Theorem  \ref{relation-GSR-GSM} (ii), we can set $f(\bx)=F_1(\bx')+\underline{\omega} F_2(\bx')$ and by
Theorem \ref{Fueter-theorem}, it holds  that for odd $q\in \mathbb{N}$,
 \begin{eqnarray*}
 \Delta_{\bx}^{\frac{q-1}{2}}f(\bx)=(q-1)!! (A(\bx')+\underline{\omega} B(\bx'))\in \mathcal{AM}(\Omega, \mathbb{A}),
\end{eqnarray*}
 where $A(\bx')$ and $B(\bx')$ are given by (\ref{q-1-formula-A-B}).\\
Hence, by Lemma  \ref{d-dbar-laplace}, we have
\begin{eqnarray*}
  \Delta_{\bx}^{\frac{q+1}{2}}f(\bx) = \Delta_{\bx} (\Delta_{\bx}^{\frac{q-1}{2}}f(\bx))  = 0.
\end{eqnarray*}
 Then by Lemma \ref{Fueter-Sce-Qian-lemma-1}, we get
 \begin{eqnarray*}
 	\frac{ \Delta_{\bx}^{\frac{q+1}{2}}f(\bx)}{(q-1)!!}=
 	\Delta_{\bx'} A(\bx')+ (q-1)   (\frac{1}{r}\partial_{r}) A(\bx')
 	+ \underline{\omega} (\Delta_{\bx'}B(\bx')+  (q-1) (\partial_{r}\frac{1}{r}) B(\bx') )=0.
 \end{eqnarray*}
By the arbitrariness of $\underline{\omega}\in\mathbb{S}$ (for example by taking $\pm\underline{\omega}$),  we deduce that
\begin{eqnarray*}
 \Delta_{\bx'} A(\bx')+ (q-1)   (\frac{1}{r}\partial_{r}) A(\bx') =0,
\end{eqnarray*}
and
\begin{eqnarray*}
 \Delta_{\bx'}B(\bx')+  (q-1) (\partial_{r}\frac{1}{r}) B(\bx')  =0.
\end{eqnarray*}
Hence, from (\ref{relation-harm-1}) and (\ref{relation-harm-2}),
$$\Delta_{\bx} A(\bx')=\Delta_{\bx}(\underline{\omega} B(\bx'))=0,$$
which, recalling (\ref{spherical-value-e-o}) and (\ref{spherical-value-derivative}),  implies that $\mathcal{PE}[\Delta_{\bx}^{\frac{q-1}{2}} f]$ and $\mathcal{PO}[\Delta_{\bx}^{\frac{q-1}{2}}f]$ are both in $ \mathcal{AH}(\Omega, \mathbb{A})$.
The proof is complete.
  \end{proof}

\begin{theorem}\label{CK-Fueter-relation-M-H}
Let $\Omega\subseteq M$ be a p-symmetric slice domain and $f:\Omega\rightarrow \mathbb{A}$ be a  generalized partial-slice monogenic function.    Then,  for odd $q\in \mathbb{N}$,
we have
 $$\mathcal{PO} \circ  \Delta_{\bx}^{\frac{q-1}{2}} f(\bx)
 =  \frac{\gamma_q }{q}HGCK [0, \Delta_{\bx_{p}}^{\frac{q-1}{2}} D_{\bx_{p}}f(\bx_{p}) ],$$
and
 $$\mathcal{PE} \circ  \Delta_{\bx}^{\frac{q-1}{2}} f(\bx)
 =\gamma_q    HGCK [  \Delta_{\bx_{p}}^{\frac{q-1}{2}} f(\bx_{p}),0].$$
which   give the  following commutative diagrams
$$\xymatrix{
 \mathcal{A}(\Omega_{0},\mathbb{A})  \ar[d]_{\frac{\gamma_q}{q}(0,  \Delta_{\bx_{p}}^{\frac{q-1}{2}} D_{\bx_{p}} )}   \ar[r]^{CK} &
 \mathcal{GSM}(\Omega,\mathbb{A}) \ar[d]^{\mathcal{PO} \circ  \Delta_{\bx}^{\frac{q-1}{2}} }
 \\
 ( \mathcal{A}(\Omega_{0},\mathbb{A}) )^{2}  \ar[d]_{q   P_2}  \ar[r]^{HGCK} & \mathcal{AH}(\Omega,\mathbb{A})  \ar[d]^{\overline{D}_{\bx}}
   \\
  \mathcal{A}(\Omega_{0},\mathbb{A})   \ar[r]^{GCK} & \mathcal{AM}(\Omega,\mathbb{A}),}
 \xymatrix{
 \mathcal{A}(\Omega_{0},\mathbb{A})  \ar[d]_{ \gamma_q (  \Delta_{\bx_{p}}^{\frac{q-1}{2}}, 0)}   \ar[r]^{CK} &
 \mathcal{GSM}(\Omega,\mathbb{A}) \ar[d]^{\mathcal{PE} \circ  \Delta_{\bx}^{\frac{q-1}{2}} }
 \\
 ( \mathcal{A}(\Omega_{0},\mathbb{A}) )^{2}  \ar[d]_{ P_1}  \ar[r]^{HGCK} & \mathcal{AH}(\Omega,\mathbb{A})  \ar[d]^{\overline{D}_{\bx}}
   \\
  \mathcal{A}(\Omega_{0},\mathbb{A})   \ar[r]^{GCK \circ \overline{D}_{\bx_p} } & \mathcal{AM}(\Omega,\mathbb{A}),} $$
  where $P_k$ is the projection of $k$-th component with $k=1,2$,   $\gamma_q$ is given by  (\ref{gamma-q}). \end{theorem}
\begin{proof}
 Let $f\in \mathcal{GSM}(\Omega,\mathbb{A})$, where $\Omega $ is a p-symmetric slice domain in $M$.
From Theorem \ref{CK-Fueter-relation-M} and Proposition  \ref{GCK-HGCK}, it holds that
  $$\mathcal{PO} \circ  \Delta_{\bx}^{\frac{q-1}{2}} f(\bx)
 =\gamma_q \mathcal{PO} \circ  GCK [ \Delta_{\bx_{p}}^{\frac{q-1}{2}}f(\bx_{p}) ]
 =  \frac{\gamma_q }{q}HGCK [0, \Delta_{\bx_{p}}^{\frac{q-1}{2}} D_{\bx_{p}}f(\bx_{p}) ],$$
and
 $$\mathcal{PE} \circ  \Delta_{\bx}^{\frac{q-1}{2}} f(\bx)
 =\gamma_q \mathcal{PE} \circ  GCK [ \Delta_{\bx_{p}}^{\frac{q-1}{2}}f(\bx_{p}) ]
 = \gamma_q   HGCK [  \Delta_{\bx_{p}}^{\frac{q-1}{2}} f(\bx_{p}),0].$$
Combining these two formulas above  with Proposition  \ref{D-GCK-HGCK},  we can get the desired  commutative diagrams.  The proof is complete.
\end{proof}

\begin{theorem}\label{CK-Fueter-relation-H}
Let $\Omega\subseteq M$ be a p-symmetric slice domain and $f:\Omega\rightarrow \mathbb{A}$ be a  generalized partial-slice monogenic function.  Then,  for odd $q\in \mathbb{N}$, we have
\begin{equation}\label{HGCK-D}
D_{\bx} \Delta_{\bx}^{\frac{q-3}{2}}f(\bx) =\gamma_q  HGCK[\Delta_{\bx_p}^{ \frac{q-3}{2}}  D_{\bx_p} f(\bx_p),0],\end{equation}
 which gives the   commutative diagram
$$\xymatrix{
 \mathcal{A}(\Omega_{0},\mathbb{A})  \ar[d]_{(\gamma_q \Delta_{\bx_{p} }^{\frac{q-3}{2}}\circ D_{\bx_p},0)}   \ar[r]^{CK} &
 \mathcal{GSM}(\Omega,\mathbb{A}) \ar[d]^{\Delta_{\bx}^{\frac{q-3}{2}} \circ D_{\bx} }
 \\
 ( \mathcal{A}(\Omega_{0},\mathbb{A}) )^{2}  \ar[d]_{\overline{D}_{\bx_p}\circ P_1 }  \ar[r]^{HGCK} & \mathcal{AH}(\Omega, \mathbb{A})  \ar[d]^{\overline{D}_{\bx}}
   \\
  \mathcal{A}(\Omega_{0},\mathbb{A})   \ar[r]^{GCK} & \mathcal{AM}(\Omega,\mathbb{A}),}
  $$
where   $\gamma_q$ is given by  (\ref{gamma-q}) and $P_1$ is the projection of the first component. \end{theorem}
 \begin{proof}
 Let $f\in \mathcal{GSM}(\Omega,\mathbb{A})$, where $\Omega $ is a p-symmetric slice domain in $M$. By  Theorems \ref{slice-Cauchy-Kovalevsky-extension} and \ref{Identity-theorem}, it holds that
$$f(\bx)=\sum_{k=0}^{+\infty} \frac{r^{2k}}{(2k)!}  (-\Delta_{\bx_p})^{k} f(\bx_p)
+\underline{\omega} \sum_{k=0}^{+\infty} \frac{r^{2k+1}}{(2k+1)!}(-\Delta_{\bx_p})^{k}(D_{\bx_p}f(\bx_p)).$$
 From the formula
  $$(\partial_{r}\frac{1}{r})^{s} r^{2k+1}=2^{s} \frac{k!}{(k-s)!}  r^{2k-2s+1}, \quad s\geq k,  $$
we  get
 \begin{eqnarray*}
 & & (\partial_{r}\frac{1}{r})^{s}   \sum_{k=0}^{+\infty} \frac{r^{2k+1}}{(2k+1)!}(-\Delta_{\bx_p}) ^{k}  D_{\bx_p} f(\bx_p)\\
&=&  2^{s} \sum_{k=s}^{+\infty} \frac{ k!   r^{2k-2s+1}}{(k-s)!(2k+1)!}(-\Delta_{\bx_p}) ^{k}  D_{\bx_p} f(\bx_p) \\
&=& 2^{s} \sum_{k=0}^{+\infty} \frac{ (k+s)!   r^{2k+1}}{ k!(2k+2s+1)!}(-\Delta_{\bx_p}) ^{k+s}  D_{\bx_p} f(\bx_p).
\end{eqnarray*}
Hence,  by Lemma  \ref{Fueter-Sce-Qian-lemma-k} for $s=\frac{q-3}{2}$, we get
\begin{equation}\label{D-k}
D_{\bx} \Delta_{\bx}^{\frac{q-3}{2}}f(\bx)= -C_{q}(\frac{q-1}{2}) 2^{\frac{q-3}{2}} \sum_{k=0}^{+\infty} \frac{ (k+\frac{q-3}{2})!   r^{2k}}{ k!(2k+q-2 )!}(-\Delta_{\bx_p}) ^{k+\frac{q-3}{2}}  D_{\bx_p} f(\bx_p),  \end{equation}
which yields
 \begin{eqnarray} \label{D-k-itself-k}
 D_{\bx} \Delta_{\bx}^{\frac{q-3}{2}}f(x) \mid_{\underline{\bx}_{q}=0}
&=&  -C_{q}(\frac{q-1}{2}) 2^{\frac{q-3}{2}} \frac{ ( \frac{q-3}{2})!   }{  ( q-2 )!}(-\Delta_{\bx_p}) ^{ \frac{q-3}{2}}  D_{\bx_p} f(\bx_p) \notag \\
&=& (-1)^{\frac{q-1}{2}} 2^{q-2} \frac{( \frac{q-1}{2})!  (\frac{q-3}{2})!}{(q-2)!}\Delta_{\bx_p}^{ \frac{q-3}{2}}  D_{\bx_p} f(\bx_p) .
\end{eqnarray}
Observing  that
$$ 2^{q-2} \frac{( \frac{q-1}{2})!  (\frac{q-3}{2})!}{(q-2)!}=2^{q-3}  (q-1) \frac{  \Big((\frac{q-3}{2})!\Big)^{2}}{(q-2)!}
= \frac{(q-1)!!}{(q-2)!!} ,$$
we can rewrite  (\ref{D-k-itself-k}) as
 \begin{equation}\label{D-k-itself-final-0}
D_{\bx} \Delta_{\bx}^{\frac{q-3}{2}}f(\bx)  \mid_{\underline{\bx}_{q}=0} =\gamma_q  \Delta_{\bx_p}^{ \frac{q-3}{2}}  D_{\bx_p} f(\bx_p),   \end{equation}
where $\gamma_q$ is given by  (\ref{gamma-q}).

Noticing  that (\ref{CK-formular-3}) in Lemma \ref{CK-formular-Four} and applying    the operator $D_{\underline{\bx}_q} $ to  (\ref{D-k})
$$D_{\bx} \Delta_{\bx}^{\frac{q-3}{2}}f(\bx)= (-1) ^{\frac{q-1}{2}}C_{q}(\frac{q-1}{2}) 2^{\frac{q-3}{2}} \sum_{k=0}^{+\infty} \frac{ (k+\frac{q-3}{2})!   \underline{\bx}_q^{2k}}{ k!(2k+q-2 )!} \Delta_{\bx_p}  ^{k}  (D_{\bx_p} f(\bx_p)),$$  we get
$$D_{\underline{\bx}_q}(D_{\bx} \Delta_{\bx}^{\frac{q-3}{2}}f(\bx))= (-1) ^{\frac{q+1}{2}}C_{q}(\frac{q-1}{2}) 2^{\frac{q-1}{2}} \sum_{k=1}^{+\infty} \frac{ (k+\frac{q-3}{2})!   \underline{\bx}_q^{2k-1}}{ (k-1)!(2k+q-2 )!} \Delta_{\bx_p}  ^{k}  (D_{\bx_p} f(\bx_p)),$$
which means
\begin{equation}\label{D-k-itself-D}
D_{\underline{\bx}_q} (D_{\bx} \Delta_{\bx}^{\frac{q-3}{2}}f(\bx)) \mid_{\underline{\bx}_{q}=0} = 0.\end{equation}
 From the harmonic generalized CK-extension (Theorem \ref{Harm-generalized-CK-extension}),  the   harmonic function $D_{\bx} \Delta_{\bx}^{\frac{q-3}{2}}f(\bx)$ is completely determined by both  its restriction to $\mathbb{R}^{p+1}$  and $D_{\underline{\bx}_q} (D_{\bx} \Delta_{\bx}^{\frac{q-3}{2}}f(\bx))$ restriction to $\mathbb{R}^{p+1}$. Hence, from (\ref{D-k-itself-final-0}) and  (\ref{D-k-itself-D}), we get
$$D_{\bx} \Delta_{\bx}^{\frac{q-3}{2}}f(\bx)=\gamma_q  HGCK [ \Delta_{\bx_p}^{ \frac{q-3}{2}}  D_{\bx_p} f(\bx_p),0 ](\bx).$$
Finally, combining    Theorem \ref{CK-Fueter-relation-M} and  Lemma \ref{d-dbar-laplace}, we can conclude the desired  commutative diagram.  The proof is complete.
\end{proof}
\begin{remark}
When $\mathbb A=\mathbb{R}_{0, m}$ with $M=\mathbb{R}^{m+1}$,  the slice monogenic version  of  Theorem \ref{CK-Fueter-relation-H} was obtained  by De Martino and  Guzm\'{a}n Ad\'{a}n  in \cite[Theorem 5.6]{De-Adan}.
\end{remark}


\section{Further remarks  and future directions of research}
 The investigations of functions with values in a real alternative $*$-algebra,  which are analytic in a suitable sense, are rather advanced in the case of slice regular functions but less developed for the classical monogenic functions. In this framework, it may be of interest to consider the more general case of functions which are poly-monogenic. Being in a non-associative setting the first question to address is a good definition of powers of an operator.
For any integer $k>1$, one may define the operator $D^{k}$ as
$$D^{k}f(x)=D(D^{k-1} f(x)), \quad f\in C^{k}(\Omega,\mathbb{A}),$$
i.e.,  recalling the notation in Section \ref{secFueterpoly},
$$D^{k}f(x)=\sum_{i_1,i_2,\ldots, i_k=0}^m  \Big(v_{i_1}v_{i_2}\ldots v_{i_k}\frac{\partial^k}{\partial_{x_{i_1}}\partial_{x_{i_2}}\ldots \partial_{x_{i_k}}}f(x) \Big)_{R}.$$
Using these notations,  we  can introduce poly-monogenic functions  over alternative $\ast$-algebras.
\begin{definition}[Poly-monogenic functions] \label{definition-poly-monogenic}
 Let $k\in \mathbb{N}\setminus \{0\}$ and $\Omega$ be an open set in $M$.    A function $f\in C^{k}(\Omega,\mathbb{A})$  is called left  poly-monogenic   of order $k$ if
$$ D^{k} f(x)= 0, \quad x\in \Omega. $$
\end{definition}
Denote by $\mathcal {PM}_k^{L}(\Omega, \mathbb{A})$ (or $\mathcal {PM}_k(\Omega, \mathbb{A})$ for short)  the function class of  all  left   monogenic functions   of order $k$.
When $k=1$, $\mathcal {PM}_k(\Omega, \mathbb{A})=\mathcal {M}(\Omega, \mathbb{A})$  coincides with the set of monogenic functions.

More generally, given  an associative order $\otimes_{(k+1)}$, one may  define
\begin{equation}\label{Dk}
(\underbrace{D D \cdots D}_{k}  f(x))_{\otimes_{(k+1)}}=
\sum_{i_1,i_2,\ldots, i_k=0}^m  \Big(v_{i_1}v_{i_2}\ldots v_{i_k} \big(\frac{\partial^k f}{\partial {x_{i_1}}\partial {x_{i_2}}\ldots \partial {x_{i_n}}}(x)\big) \Big)_{\otimes_{(k+1)}}.\end{equation}
Note that, for $f\in C^{k}(\Omega,\mathbb{A}) $, on the right hand side of (\ref{Dk}) we can group all the derivatives of the form
  $\partial^{\mathrm{k}}f=\frac{\partial^{k} f}{\partial_{x_{0}}^{k_0}\partial_{x_{1}}^{k_1}\cdots\partial_{x_{m}}^{k_m} }$ with $\mathrm{k}=(k_0,k_1,\ldots,k_{m})\in \mathbb{N}^{m+1}$ and $k=|\mathrm{k}|$ as
  \begin{eqnarray}\label{sum-two}
 && \sum_{i_1,i_2, \ldots, i_k=0}^m \Big(v_{i_1}v_{i_2}\ldots v_{i_k} \big(\frac{\partial^k f}
 {\partial {x_{i_1}}\ldots \partial {x_{i_k}}} (x) \big) \Big)_{\otimes_{(k+1)}}\notag
 \\
 &=&\sum_{|\mathrm{k}|=k}  \sum_{(i_1,i_2, \ldots, i_k)\in \sigma(\overrightarrow{\mathrm{k}})  }
  (v_{i_1}  v_{i_2} \cdots  v_{i_k}  \partial^{\mathrm{k}} f(x))_{\otimes_{(k+1)}},
\end{eqnarray}
where $\sigma(\overrightarrow{\mathrm{k}})$ is as in Definition \ref{definition-Fueter}.

 Hence,  by Proposition \ref{no-order}, $(\underbrace{D D \cdots D}_{k}  f(x))_{\otimes_{(k+1)}}$ does not depend on the   associative order $\otimes_{(k+1)}$ and then we can write (\ref{Dk}) as $D^k f(x)$ as there is  no ambiguity.

\begin{example}\label{k-Cauchy-kernel-example}
Let $\Omega=M\setminus \{0\}$ and define, for $k=1,2,\ldots$,
$$E^{[k]}(x)=\frac{1}{\sigma_{m}}\frac{1}{(k-1)!}\frac{\overline{x}x_{0}^{k-1}}{|x|^{m+1}}, \quad x \in\Omega,$$
where $\sigma_{m}=2\frac{\Gamma^{m+1}(\frac{1}{2}) }{\Gamma (\frac{m+1}{2})} $ is the surface area of the unit ball in $\mathbb{R}^{m+1}$. \\
One may readily check that
$$D^{n}E^{[k]}=E^{[k]}D^{n} =E^{[k-n]}, \quad n=1,2,\ldots,k,$$
and
$$D^{k} E^{[k]}=E^{[k]}D^{k} =0,$$
which  shows that $E^{[k]} \in \mathcal {PM}_{k}(\Omega, \mathbb{A})$.
\end{example}

\begin{remark}
 When $ \mathbb{A}=\mathbb{O}=M$,  the operator $D$ is power-associative in the sense of \cite[Definition 4.5]{HRSX};  see also \cite[Proposition 8.3]{HRSX}.
Now we give a direct proof of this fact by showing  that $D^k $ right octonionic para-linear for all $k\in\mathbb{N}$. First of all, we observe that  (\ref{sum-two}) gives that, for all $a\in \mathbb{O}$,
 \[
  \begin{split}
 D^k (f(x)a)=& \sum_{|\mathrm{k}|=k}  \sum_{(i_1,i_2, \ldots, i_k)\in \sigma(\overrightarrow{\mathrm{k}})  }
  (v_{i_1}  v_{i_2} \cdots  v_{i_k}  (\partial^{\mathrm{k}}f(x)a))_{L}\\
  =& \sum_{|\mathrm{k}|=k} \Big( \sum_{(i_1,i_2, \ldots, i_k)\in \sigma(\overrightarrow{\mathrm{k}})  }
  (v_{i_1}  v_{i_2} \cdots  v_{i_k})_{L} \Big)(\partial^{\mathrm{k}}f(x)a).\end{split}
  \]
  where $\sigma(\overrightarrow{\mathrm{k}})$ is as in Definition \ref{definition-Fueter}.

Hence, recalling the fundamental formulas
$${\rm{Re}}\,( a+b)={\rm{Re}}\,  a +{\rm{Re}}\,  b, \quad {\rm{Re}}\,((ab)c)={\rm{Re}}\, (a(bc)), \quad a, b, c\in \mathbb{O},$$
 we have, for $f\in C^{k}(\Omega,\mathbb{O})$ and  $a\in \mathbb{O}$,
    \[
  \begin{split}
{\rm{Re}}\, (D^k (f(x)a))
  =& \sum_{|\mathrm{k}|=k}{\rm{Re}}\, \Big( \Big( \sum_{(i_1,i_2, \ldots, i_k)\in \sigma(\overrightarrow{\mathrm{k}})  }
  (v_{i_1}  v_{i_2} \cdots  v_{i_k})_{L} \Big)(\partial^{\mathrm{k}}f(x)a) \Big)\\
   =& \sum_{|\mathrm{k}|=k}{\rm{Re}}\, \Big( \Big(\Big(\sum_{(i_1,i_2, \ldots, i_k)\in \sigma(\overrightarrow{\mathrm{k}})  }
  (v_{i_1}  v_{i_2} \cdots  v_{i_k})_{L} \Big) \partial^{\mathrm{k}} f(x) \Big) a\Big)\\
  =& \sum_{|\mathrm{k}|=k}{\rm{Re}}\, \Big( \Big(\sum_{(i_1,i_2, \ldots, i_k)\in \sigma(\overrightarrow{\mathrm{k}})  }
  ( (v_{i_1}  v_{i_2} \cdots  v_{i_k})_{L} \partial^{\mathrm{k}} f(x)  )  \Big) a\Big)\\
   =& \sum_{|\mathrm{k}|=k}{\rm{Re}}\, \Big( \Big(\sum_{(i_1,i_2, \ldots, i_k)\in \sigma(\overrightarrow{\mathrm{k}})  }
  (v_{i_1}  v_{i_2} \cdots  v_{i_k}\partial^{\mathrm{k}} f(x))_{L}   \Big) a\Big)\\
  =& \sum_{|\mathrm{k}|=k}{\rm{Re}}\, \Big( \Big(\sum_{(i_1,i_2, \ldots, i_k)\in \sigma(\overrightarrow{\mathrm{k}})  }
  (v_{i_1}  v_{i_2} \cdots  v_{i_k}\partial^{\mathrm{k}} f(x) )_{L} \Big) a\Big)\\
= &{\rm{Re}}\, ((D^kf(x)) a),  \end{split}
  \]
which allows us to conclude that all  operators   $D^k$ are right octonionic para-linear for all $k$.
 \end{remark}

\begin{remark}
Note  that  (\ref{sum-two}) relies on the fact that the partial derivatives  $\partial_{x_{s}} $  commute each other  on $\mathcal{C}^k(\Omega,\mathbb{A})$  for $s=0,1,\ldots,m$, which still holds for   Dunkl operators  associated with the Coxeter group  and the multiplicity function, i.e.
  $\mathfrak{D}_s \mathfrak{D}_t =\mathfrak{D}_t \mathfrak{D}_s $. For the precise definition of $\mathfrak{D}_s$, see for example \cite{Dunkl}. Hence, we can  define poly-Dunkl-monogenic   functions which are in the kernel $\mathfrak{D}^{k}$, where $\mathfrak{D}$ is the so-called Dunkl-Dirac operator given by
 $$\mathfrak{D}f(x)=\sum _{s=0}^{m}v_{s} \mathfrak{D}_sf(x)= \sum _{s =0}^{m} \sum _{t=0}^{d}   v_{s}v_{t} \mathfrak{D}_sf_t(x), \quad x\in \Omega, $$
 where $f=\sum_{t=0}^{d}  v_{t}   f_{t} \in C^{1}(\Omega,\mathbb{A})$.
  \end{remark}
See  \cite{CKR-06} for $k=1$ and $\mathbb A=\mathbb{R}_{0, m}$ with $M=\mathbb{R}^{m+1}$.
Besides,  we  can also define  the class  of   generalized partial-slice poly-monogenic functions   (or poly generalized  partial-slice monogenic functions).
\begin{definition}[Generalized  partial-slice poly-monogenic functions] \label{definition-poly-partial-slice-monogenic}
 Let $k\in \mathbb{N}\setminus \{0\}$ and $\Omega$ be an open set in $M$. A function $f :\Omega \rightarrow \mathbb{A}$ is called left generalized  partial-slice poly-monogenic   of type $(p,q)$ and order $k$ if, for all $\underline{\omega} \in \mathbb S$, its restriction $f_{\underline{\omega}}\in C^{k}(\Omega_{\underline{\omega}}, \mathbb{A})$   satisfies
$$D_{\underline{\omega}}^{k}f_{\underline{\omega}}(\bx_p+r\underline{\omega})=(D_{\bx_p}+\underline{\omega} \partial_{r})^{k} f_{\underline{\omega}}(\bx_p+r\underline{\omega})=0,$$
for all $\bx=\bx_p+r\underline{\omega} \in  \Omega_{\underline{\omega}}$.
 \end{definition}

To conclude the paper, we give an example of left (and right) generalized  partial-slice poly-monogenic   of type $(p,q)$ and order $k$. More properties of generalized partial-slice poly-monogenic functions over alternative $\ast$-algebras will be explored in a future work.
\begin{example}\label{k-Cauchy-kernel-example}
Let $\Omega=M\setminus \{0\}$ and define, for $k=1,2,\ldots$,
$$E^{[k]}(\bx)=\frac{1}{\sigma_{p+1}}\frac{1}{(k-1)!}\frac{\overline{\bx}x_{0}^{k-1}}{|\bx|^{p+2}}, \quad \bx \in\Omega,$$
where $\sigma_{p+1}=2\frac{\Gamma^{p+2}(\frac{1}{2}) }{\Gamma (\frac{p+2}{2})} $ is the surface area of the unit ball in $\mathbb{R}^{p+2}$.
\end{example}

\par
\noindent
\textbf{Author contributions}
All authors have contributed equally to all aspects of this manuscript and have reviewed its final draft.
\\
\textbf{Conflict of interest}
There is no financial or non-financial interests that are directly or indirectly related to the work submitted for publication.
\\
\textbf{Data availability}
Data sharing is not applicable to this article as no datasets were generated  during the current study.




\vskip 10mm
\end{document}